\documentclass[a4paper,12pt]{amsart}

\usepackage{amsthm}
\usepackage{amsmath}
\usepackage{amssymb}
\usepackage{latexsym}
\usepackage{enumerate}

\usepackage{graphicx}

\usepackage{bm}

\renewcommand{\thetheoremName}

\newtheorem{theorem}{Theorem}[section]

\newtheorem{proposition}[theorem]{Proposition}

\theoremstyle{definition}
\newtheorem{definition}[theorem]{Definition}
\newtheorem{example}[theorem]{Example}

\newtheorem{remark}[theorem]{Remark}

\numberwithin{equation}{section}

\newcommand{\Hess}{\operatorname{Hess}}

\newcommand{\spanning}{\operatorname{Span}}

\begin{document}

\title[Frozen Finsler metrics]
{Geodesic sprays and frozen metrics\\
in rheonomic Lagrange manifolds}

\author{Steen Markvorsen}
\address{DTU Compute, Mathematics, Kgs. Lyngby, Denmark}
\email[S. Markvorsen]{stema@dtu.dk}

\subjclass[2000]{Primary 53, 58}


\begin{abstract}
We define systems of pre-extremals for the energy functional of regular rheonomic Lagran\-ge manifolds and show how they induce well-defined Hamilton orthogonal nets. Such nets have applications in the modelling of e.g. wildfire spread under time- and space-dependent conditions.  The time function inherited  from such a Hamilton net induces in turn a time-independent Finsler metric -- we call it the associated frozen metric. It is simply obtained by inserting the time function from the net into the given Lagrangean. The energy pre-extremals then become ordinary Finsler geodesics of the frozen metric and the Hamilton orthogonality property is preserved during the freeze.
We compare our results with previous findings of G. W. Richards concerning his application of Huyghens' principle to establish the PDE system for Hamilton orthogonal nets in 2D Randers spaces and also concerning his explicit spray solutions for time-only dependent Randers
spaces. We analyze examples of time-dependent 2D Randers spaces with simple, yet non-trivial, Zermelo data; we obtain analytic and numerical solutions to their respective energy pre-extremal equations; and we display details of the resulting (frozen) Hamilton orthogonal nets.
\end{abstract}

\maketitle



\section{Introduction} \label{secIntroduction}

A large number of natural phenomena evolve under highly non-isotropic and time-varying conditions. The spread of wildfires is but one such phenomenon -- see \cite{Markvorsen2016}. The concept of a regular rheonomic Lagrange manifold $(M, L)$ offers a natural global geometric setting for an initial study of such phenomena. The anisotropy as well as the time- and space-dependency is represented by a time-dependent Finsler metric $F$ with $L=F^{2}$. As a further structural Ansatz for the phenomena under consideration we will assume that they are 'driven' in this Finsler metric background as wave frontals issuing from a given initial base hypersurface $N$ in $M$ with $F$-unit speed rays, which all leave $N$ orthogonally w.r.t. the metric $F$. Huyghens' principle then implies that the rays are everywhere what we call Hamilton orthogonal to the frontals. This key observation was also worked out by Richards in \cite{Richards1990} and \cite{Richards1993a}, where he presents an explicit PDE system which is equivalent to the Hamilton orthogonality for the special 2D Finsler manifolds known as 2D Randers spaces, represented by their so-called Zermelo data. \\

Our main result is that in general, i.e for any regular Lagrangean, and under the structural Ansatz above, Hamilton orthogonality is equivalent to the condition that all the rays of the spread phenomenon are energy pre-extremals for the given Lagrangean. The first order PDE system for Hamilton orthogonality is thus equivalent to a second order ODE system for these pre-extremals. The spread problem is in this way solvable via the rays, which then together -- side by side -- mold the frontals of the spread phenomenon.\\

As a corollary -- which may be of interest in its own right -- we also show the following. If we freeze the metric $F$ to the specific functional value that it has at a given point precisely when the frontal of the given spread passes through this point, then we obtain a \emph{frozen} time-independent Finsler metric $\widehat{F}$, in which the given rays are  $\widehat{F}$-geodesics, which mold the same frontals as before. These frozen metrics are highly dependent on both the base hypersurface $N$ and on the time of ignition from $N$.\\

As already alluded to in the abstract we illustrate and support these main results by explicit calculations for simple 2D time-dependent Randers metrics and we display various details from the corresponding point-ignited spread phenomena.


\subsection{Outline of paper}
In section \ref{secRheonomicLagrange} we describe the concept of a regular rheonomic Lagrange manifold and introduce the corresponding time-dependent indicatrix field. For any given variation of a given curve the corresponding $L$-energy and $F$-length (for unit speed curves) is differentiated with respect to the variation parameter
in section \ref{secVarEnergy} and the ensuing extremal equations are displayed. The notion of a unit fiber  net is introduced in section
\ref{secWFnet} as a background for the presentation of the main results in sections
\ref{secMainThms} and \ref{secFrozen} concerning the equivalence of Hamilton orthogonality and energy pre-extremal rays and concerning the frozen metrics -- as mentioned in the introduction above. The rheonomic Randers spaces and their equivalent Zermelo data are considered in section \ref{secRanders} with the purpose of presenting Richards' results and the promised examples in section \ref{secRichards}.

\section{Rheonomic Lagrange manifolds} \label{secRheonomicLagrange}
A regular rheonomic Lagrange space $(M^{n}, L)$ is a smooth manifold $M$ with a time-dependent Lagrangean $L$ modelled on a time-dependent Finsler metric $F$, i.e. $L = F^{2}$  -- see e.g.
\cite{AZ1}, \cite{Anastasiei1994},
\cite{munteanu2003a}, \cite{frigioiu2008a}, and
\cite{trumper1983a}. The metric $F$ induces for each time $t$ in a given time interval a smooth family of Minkowski norms in the tangent spaces of $M$. For the transparency of this work, we shall be mainly interested in two-dimensional cases and examples. Correspondingly we write -- with $M = \mathbb{R}^{2}$, $p = (u,v) = (u^{1}, u^{2})$, and $V = (x,y) = (x^{1}, x^{2}) \in T_{p}\mathbb{R}^{2} = \spanning \{\partial_{u} , \partial_{v} \}$:
\begin{equation}
\begin{aligned}
F = F_{t} = F(t, p) &= F(t, p, V) \\
&= F(t, u,v, x, y) \\
&=  F(t, u^{1}, u^{2}, x^{1}, x^{2})\quad ,
\end{aligned}
\end{equation}
where the latter index notation is primarily used here for expressions involving general summation over repeated indices. The higher dimensional cases are easily obtained by extending the indices beyond $2$ -- like $p = (u^{1}, \cdots , u^{n})$ and $V = (x^{1}, \cdots ,  x^{n}) \in T_{p}M$.\\

By definition, see \cite{ShenBook2001} and \cite{BCS}, a Finsler metric on a domain $\mathcal{U}$ is a smooth family of Minkowski norms on the tangent planes, i.e. a smooth  family of indicatrix templates which at each time $t$ in each tangent plane $T_{p}\,\mathcal{U}$ at the respective points $p = (u,v)$ in the parameter domain $\mathcal{U}$  is determined by the nonnegative smooth function $F_{t}$ of $t$ as follows:
\begin{enumerate}
\item $F_{t}$ is smooth on each punctured tangent plane $T_{p}\,\mathcal{U} - \{(0,0)\} $.
\item $F_{t}$ is positively homogeneous of degree one: $F_{t}(kV) = kF_{t}(V)$ for every $V\in T_{p}\,\mathcal{U}$ and every $k > 0$.
\item The following bilinear symmetric form on the tangent plane is positive definite:
    \begin{equation}\label{eqPosDef}
    g_{t,p, V}(U, W) = \frac{1}{2}\frac{\partial^{2}}{\partial \lambda \partial \mu}[F_{t}^{2}(V + \lambda U + \mu W)]_{|\lambda=\mu=0}
    \end{equation}
\end{enumerate}

Since the function $F_{t}$ is homogenous of degree $1$, the fundamental metric $g_{t,p, V}(U, W)$ satisfies the following for each time $t$:
\begin{eqnarray}
g_{t,p, V}(V,W) &=& \frac{1}{2}\frac{\partial}{\partial \lambda }[F_{t}^{2}(V + \lambda W)]_{|\lambda=0}\\
g_{t,p, V}(V,V) &=& F_{t}^{2}(V) = \Vert V \Vert^{2}_{F_{t}} \quad .
\end{eqnarray}

Suppose that we use the canonical basis $\{\partial_{u} = b_{1}, \partial_{v} = b_{2}\}$ in $T_{p}\,\mathcal{U}$, and let $V = x^{i}b_{i}$.
Then we can define coordinates of $g = g_{t,p,V}$ in the usual way:
\begin{eqnarray}
2g_{i j}(V) &=& 2g_{t,p, V}(b_{i}, b_{j}) \\
 &=& \frac{\partial^{2}}{\partial \lambda \partial \mu}[F_{t}^{2}(V + \lambda \,b_{i} + \mu \,b_{j})]_{|\lambda=\mu=0} \\
 &=& \Hess_{i \, j}(F_{t}^{2})(V)\\
 &=& [F_{t}^{2}]_{x^{i} x^{j}}(V) \quad ,
\end{eqnarray}
where the Hessian is evaluated at the vector $V$ and where the last line $[F_{t}^{2}]_{x^{i} x^{j}}$  is 'shorthand' for the double derivatives of $F_{t}^{2}$ with respect to the tangent plane coordinates $x^{i}$.

In the following we shall need other partial derivatives of $F_{t}^{2}$ -- such as $[F_{t}^{2}]_{t}(V)$, $[F_{t}^{2}]_{u^{k}}(V)$, and $[F_{t}^{2}]_{u^{l} x^{k}}(V)$  -- as well as the inverse matrix of $g_{i j}(V)$, which are now all well-defined, e.g.:
\begin{equation}
[g^{i j}(V)] = [ g_{i j}(V) ]^{-1} \, \, \, \, \textrm{and} \, \, \, \, [g^{i j}(V)g_{k j}(V)]  = \left[
            \begin{array}{cc}
              1 & 0 \\
              0 & 1 \\
            \end{array}
          \right]
 \quad .
\end{equation}
Moreover, the following second order informations are of well-known and instrumental importance for the study of Finsler manifolds and Lagrangean geometry -- see \cite{ShenBook2001}, \cite{BCS}, \cite{AIM}, and \cite{trumper1983a}:
\begin{equation}
\begin{aligned}
G^{i}(y) &= \left(\frac{1}{4}\right)g^{i\,l}(y)\left(\left[ F_{t}^{2}\right]_{x^{k}\,y^{l}}(y)y^{k} - \left[ F_{t}^{2}\right]_{x^{l}}(y) \right) \\
N_{0}^{i}(y) &=  \left(\frac{1}{2}\right)g^{i\,l}(y)\left[ F_{t}^{2}\right]_{t\,y^{l}}(y) \quad .
\end{aligned}
\end{equation}

\begin{definition}\label{defIndicatrix}
The set of points in the tangent plane $T_{p}\,\mathcal{U}$ which have $F_{t}$-unit position vectors is called the instantaneous indicatrix of $F_{t}$ at $p$:
\begin{equation}
\mathcal{I}_{t, p} = F^{-1}_{t}(1) = \{V \in T_{p}\,\mathcal{U}\, | \, F(t,p, V) = 1\} \quad .
\end{equation}
\end{definition}

Since $g_{t,p, V}$ is positive definite, every  indicatrix $\mathcal{I}_{t, p}$ is automatically strongly convex in its tangent plane at $p$, and it contains the origin of the tangent plane in its interior, see \cite{BCS}. It is therefore a pointed oval -- the point being that origin of the tangent plane.\\

\section{Variations of $L$-energy and of $F$-length} \label{secVarEnergy}

Since we shall be interested in particular aspects of the pre-extremals of the energy functional in $(M, L)$ we briefly review the first variation of energy --  with special emphasis on the influence of the time dependence of the underlying metric. In the following we suppress the indication of the time-dependence and write  $F$  for $F_{t}$.

The first variation formula will give the ODE differential equation conditions for a curve to be an $F^{2}$- \emph{energy extremal} in $M$. The ODE system for the extremals are, of course, nothing but the Euler--Lagrange equations for the time-dependent Lagrange functional $L$, see \cite[(3.5), (3.6)]{Anastasiei1994} and \cite{Antonelli1991}: \\

We let $c$ denote a candidate curve for an extremal of $F^{2}$, i.e. of $L$ :
\begin{equation}
c \quad : \quad [a, b] \to M
\end{equation}
This means that
there is a partition of $[a,b]$
\begin{equation}
a = t_{0} < \, \cdots \, < t_{m} = b \quad ,
\end{equation}
such that $c$ is smooth on each subinterval $[t_{i-1}, t_{i}]$ for every $i = 1, \, \cdots \, , m$ quad .

A \emph{variation} of the curve $c$ is then a piecewise smooth map

\begin{equation}
H \quad : \quad (-\varepsilon, \varepsilon) \times [a,b] \to M
\end{equation}
such that
\begin{equation} \label{eqVar}
\begin{cases}
&H \quad \textrm{is continuous on} \quad (-\varepsilon, \varepsilon) \times [a,b] \\
&H \quad \textrm{is smooth  on each} \quad (-\varepsilon, \varepsilon) \times [t_{i-1}, t_{i}] \\
&H(0,t) = c(t) \quad \textrm{for all} \quad a \leq t \leq b \quad .
\end{cases}
\end{equation}
The last equation in (\ref{eqVar}) states that $c$ is the base curve in the family of curves $c_{u}(t) = H(u,t)$, which sweeps out the variation.

The variation $H$ induces the associated \emph{variation vector field} $V(t)$, so that we have, in local coordinates:
\begin{equation}
\frac{\partial H}{\partial u}(0, t) = V(t) = V^{k}(t)\frac{\partial}{\partial x_{k}}{|_{c(t)}}
\end{equation}

The $F^{2}$-energy values of the individual piecewise smooth curves $c_{u}(t)$ in the variation family $H$ are then given by
\begin{equation}
\begin{aligned}
\mathcal{E}(u) &= \int_{a}^{b}\, L\left(t, c_{u}(t), \dot{c}_{u}(t)\right) \, dt \\
&= \int_{a}^{b}\, F^{2}\left(t, c_{u}(t), \dot{c}_{u}(t)\right) \, dt \\
&= \sum_{i=1}^{m} \int_{t_{i-1}}^{t_{i}}\, F^{2}\left(t, c_{u}(t),  \frac{\partial H}{\partial t}(u, t) \right)\, dt
\end{aligned}
\end{equation}

Then we have the following $u$-derivative of $\mathcal{E}(u)$ at $u=0$.  We refer to \cite{BCS} and apply the short hand notation presented in section \ref{secRheonomicLagrange}. This calculation mimics almost verbatim the classical calculation in \cite{ShenBook2001} with the difference, however, that our $F$ field is now time-dependent, so that there will be an explicit extra term in the integrand below. This extra term is precisely given by $N_{0}^{i}(y)$.

\begin{equation*}
\begin{aligned}
\mathcal{E}'(0) &= \int_{a}^{b}\, \left(\left[F^{2}\right]_{x^{k}}V^{k} + \left[F^{2} \right]_{y^{k}}\frac{dV^{k}}{dt} \right) \, dt \\
&= \int_{a}^{b}\, \left(\left[ F^{2}\right]_{x^{k}} - \left(\frac{d}{dt}\left[F^{2} \right]_{y^{k}} \right) \right)V^{k} \, dt \\
&\phantom{aaaaaaa}+ \sum_{i=1}^{m} \left[\left[F^{2} \right]_{y^{k}} V^{k}\right]_{t_{i-1}}^{t_{i}}\\
&=  \int_{a}^{b}\,\left(\left[ F^{2}\right]_{x^{k}} - \left[F^{2}\right]_{t\,y^{k}} - \left[ F^{2}\right]_{x^{l}\,y^{k}}\dot{c}^{\,l} -  \left[ F^{2}\right]_{y^{l}\,y^{k}}\,\ddot{c}^{\,l} \right)V^{k} \, dt \\
&\phantom{aaaaaaa}+\sum_{i=1}^{m} \left[g_{j\,k}\,\dot{c}^{\,j}\, V^{k}\right]_{t_{i-1}}^{t_{i}}\\
&= - 2\int_{a}^{b}\,g_{j\,k}\,\left(\ddot{c}^{\,j} + 2G^{j}(\dot{c}) + N_{0}^{j}(\dot{c}) \right)V^{k} \, dt \\
&\phantom{aaaaaaa}+2\sum_{i=1}^{m} \left[g_{j\,k}\,\dot{c}^{\,j}\, V^{k}\right]_{t_{i-1}}^{t_{i}} \quad .
\end{aligned}
\end{equation*}

We have thus

\begin{theorem}\label{thmLagrangeExtremal}
Let  $H$ and $V$ denote a variation of a curve $c$  as above.
Then the energy functional on the given variation is
\begin{equation}
\mathcal{E}(c) = \int_{a}^{b} F^{2}\left(t, c_{u}(t), \dot{c}_{u}(t)\right) \, dt \quad ,
\end{equation}
with the following derivative:
\begin{equation} \label{eqEderiv}
\begin{aligned}
\mathcal{E}'(0) &= - 2\int_{a}^{b}\,g_{j\,k}\,\left(\ddot{c}^{\,j} + 2G^{j}(\dot{c}) + N_{0}^{j}(\dot{c}) \right)V^{k} \, dt \\
&\phantom{aaaaaaa}+2\sum_{i=1}^{m} \left[g_{j\,k}\,\dot{c}^{\,j}\, V^{k}\right]_{t_{i-1}}^{t_{i}} \quad .
\end{aligned}
\end{equation}
\end{theorem}

Since we shall need it below we observe the following immediate analogue for the $F$-\emph{length} functional $\mathcal{L}$ under the assumption that the base curve $c$ is $F$-unit speed parametrized, see \cite{ShenBook2001}, \cite{Markvorsen2016}:

\begin{proposition}
Suppose again we conider the variational setting with $H$ and $V$ as above, but now with base curve $c$ satisfying the unit speed condition
\begin{equation}
\Vert \dot{c}(t) \Vert_{F} = 1 \quad \textrm{for all} \quad t \in \, [a, b]\quad.
\end{equation}
Then the length functional on the given variation is
\begin{equation}
\mathcal{L}(c) = \int_{a}^{b} F\left(t, c_{u}(t), \dot{c}_{u}(t)\right) \, dt \quad ,
\end{equation}
and it has the same expression for its derivative as the energy functional -- except for the factor $2$:
\begin{equation} \label{eqLderiv}
\begin{aligned}
\mathcal{L}'(0) &=
- \int_{a}^{b}\,g_{j\,k}\,\left(\ddot{c}^{\,j} + 2G^{j}(\dot{c}) + N_{0}^{j}(\dot{c}) \right)V^{k} \, dt \\
 &\phantom{aaaaaaa}+\sum_{i=1}^{m} \left[g_{j\,k}\,\dot{c}^{\,j}\, V^{k}\right]_{t_{i-1}}^{t_{i}} \quad .
\end{aligned}
\end{equation}
\end{proposition}

\begin{remark}
We note that if the Lagrange-Finsler metric $F$ is actually time-independent, then the formulas \eqref{eqEderiv} and \eqref{eqLderiv} above are just modified by setting $N_{0}^{j}(\dot{c}) =0$.
\end{remark}
\section{Unit fiber nets} \label{secWFnet}

We first define a special class of nets in $(M^{n}, F)$ as follows:

\begin{definition} \label{defNet}
Let $N^{n-1}$ be a smooth embedded orientable hypersurface in $(M^{n}, F)$ with a well-defined choice of normal  vector field $n$ at time $t_{0}$, i.e. $n$ is everywhere orthogonal to $N$ with respect to $F_{t_{0}}$ and has unit length with respect to $F_{t_{0}}$.
An $(N, t_{0})$-based \emph{unit fiber net} $\gamma$ in $(M^{n}, F)$ is a diffeomorphism
\begin{equation}
\gamma \, \, : \, \, N\, \times \,\, ]\,t_{0}, T[ \quad \longmapsto \quad  \mathcal{U} \subset M
\end{equation}
with the property that each fiber $\gamma(s_{0}, t)$ is a unit speed ray that 'leaves' $N$ orthogonally at time $t_{0}$:
\begin{equation} \label{eqWFray}
\begin{aligned}
\gamma(s, t_{0}) &= s  \quad \textrm{for all} \quad s \in N \quad \textrm{and} \\
\dot{\gamma}_{t}(s, t_{0}) &=  n(s) \quad \textrm{for all} \quad s \in N \quad \textrm{and} \\
\Vert \dot{\gamma}_{t}(s, t)\Vert &=1 \quad \textrm{for all} \quad s \in N \quad \textrm{and all} \quad t \in \, ]t_{0}, T[ \quad .
\end{aligned}
\end{equation}
\end{definition}

\begin{remark}
The definition of an $(N, t_{0})$-based \emph{unit fiber net}  can easily be extended to submanifolds $N$ of higher co-dimension than $1$ by fattening the submanifold to a sufficiently thin $\varepsilon$-tube w.r.t. $F_{t_{0}}$ and then consider the boundary hypersurface of that tube instead. We shall tacitly assume this construction for the cases illustrated explicitly below, where $N$ is a point $p$ -- in which case the point in question  should be fattened to a small $\varepsilon$-sphere around $p$. For example, a $(p, t_{0})$-based \emph{unit fiber net} in $\mathbb{R}^{2}$ which corresponds to axis-symmetric polar elliptic coordinates is then a parametrization of the following type $\gamma(s,t) = (p_{1} + t\cdot a\cdot \cos(s), p_{2} + t\cdot b \cdot \sin(s))$, $t \in \, ]0, \infty[$, $s \in \, \mathbb{S}^{1}$. These specific  nets  appear naturally in constant Randers metrics in $\mathbb{R}^{2}$ with Zermelo data $a$, $b$, $C=0$, and $\theta=0$ -- see sections \ref{secRanders} and \ref{secRichards} below.
\end{remark}

\subsection{Huyghens' principle}
A much celebrated and useful principle for obtaining a particular spray structure of an  $(N, t_{0})$-based \emph{unit fiber net}, which  is also our main concern in this work, is Huyghens' principle -- see  \cite{Arnold1989} and \cite{Markvorsen2016}. To state the principle together with one of its significant consequences, we shall first introduce the time-level sets of the given net. For each $t_{1} \in \, ]t_{0}, T[$ we define such a  level in the usual way:
\begin{equation}
\eta_{t_{1}} =  \{ p \in \mathcal{U} \, | \,  t(p) = t_{1}\, \} \quad .
\end{equation}

\begin{definition} \label{defHuyghensPrincip}
Let $\gamma(s,t)$ denote an $(N, t_{0})$-based unit fiber net. We will say that the net satisfies Huyghens' principle if the following holds true for all sufficiently small (but nonzero) time increments $\delta$: The time level set $\eta_{t_{1} + \delta}$ is obtained as the \emph{envelope} of the set of $(p, t_{1})$-based unit fiber nets of duration $t \in \, ]t_{1}, t_{1} + \delta[$ in $\mathcal{U}$ where $p$ goes through all points in $\eta_{t_{1}}$. These $(p, t_{1})$-based unit fiber nets will be called Huyghens droplets of duration $\delta$ from the level set $\eta_{t_{1}}$ -- see examples in figure \ref{figHuyghens34} below.
\end{definition}

The structural consequence of Huyghens' principle is then encoded into the following result, see \cite[Theorem p. 251]{Arnold1989}, and compare \cite[figure 196]{Arnold1989} (concerning conjugate directions) with figure \ref{figHamilton34} (concerning the equivalent $F$-orthogonal directions) below.

\begin{theorem} \label{thmHuyghens}
Let $\gamma(s,t)$ denote an $(N, t_{0})$-based unit fiber net which satisfies Huyghens' principle. Then  the direction $\dot{\gamma}_{s}(s,t_{1})$ of the frontal $\eta_{t_{1}}$ at time $t_{1}$ and at a given point $p$ on that frontal is $F$-orthogonal to the direction $\dot{\gamma}_{t}(s,t_{1})$ of the ray  $\gamma(s,t)$ through $p$, i.e.
\begin{equation} \label{eqHuyghensOrthog}
\dot{\gamma}_{s}(s,t) \bot_{F_{t}} \dot{\gamma}_{t}(s,t) \quad \textrm{for all} \quad s \in N \, , \, t\, \in \, ]t_{0}, T[  \quad. \\
\end{equation}
\end{theorem}

\begin{remark}
The orthogonality obtained and expressed in \eqref{eqHuyghensOrthog} we will call \emph{Hamilton orthogonality}. In fact, the proof for Hamilton orthogonality only needs an infinitesimal version of Huyghens' principle.
\end{remark}

This specific consequence of Huyghens principle was established for $(N, t_{0})$-based unit fiber nets in 2D Randers spaces by G. W. Richards in \cite{Richards1990}, where the ensuing Hamilton orthogonality is expressed directly in terms of a system of PDE equations -- see theorem \ref{thmRichardsHamilton} in section \ref{secRichards} below.

\begin{remark} \label{remHuyghensRev}
It is not clear if -- or under which additional conditions -- the converse to theorem \ref{thmHuyghens} holds true, i.e. if  Hamilton orthogonality for a given $(N, t_{0})$-based unit fiber net implies that the net and the background metric also satisfies Huyghens' principle? The problem is that the metric $F$ is in this generality time-dependent so that e.g. the triangle inequality does not hold for the respective rays. As we shall see below, however, the rays do solve an energy pre-extremal problem, and the rays are actually geodesics in the so-called frozen metrics associated with $F$. However, these frozen metrics typically do not agree if the corresponding nets have different base hypersurfaces $N$. Nevertheless, it is indicated by figure \ref{figHuyghens34} in section \ref{secRichards} below, that in suitable simplistic settings, like the ones under consideration there, the
Huyghens droplets from one frontal actually seem to envelope the next frontal in a corresponding increment of time.

\begin{figure}[h]
 \includegraphics[width=70mm]{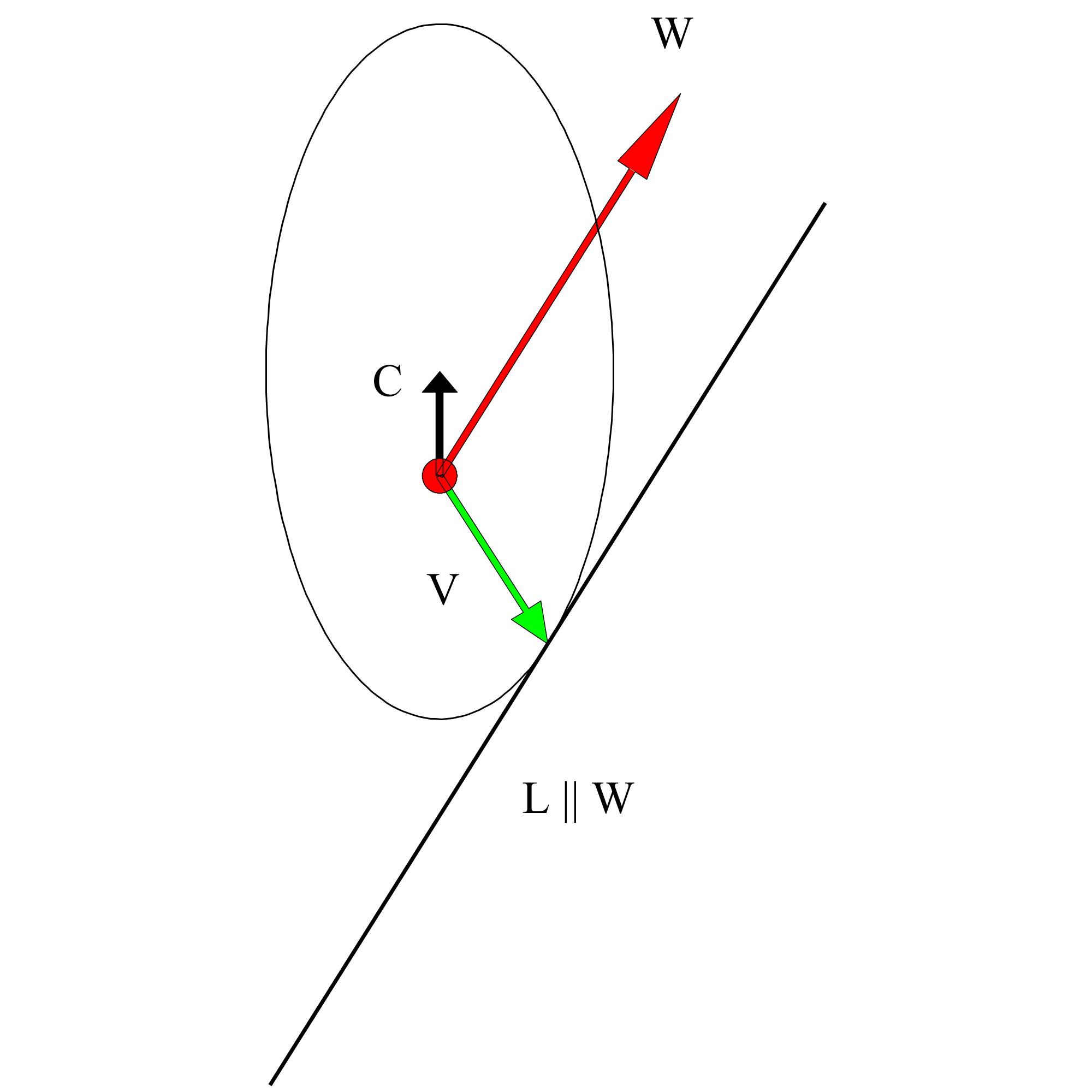}
\begin{center}
\caption{A typical elliptic indicatrix for a Randers metric $F$ at a given point for a given time. The two vectors $V$ and $W$ are $F$-orthogonal -- see section \ref{secRanders} below and \cite{Markvorsen2016}.} \label{figHamilton34}
\end{center}
\end{figure}

 If $F$ is time-independent we do have the necessary triangle inequalities at our disposal, and in this case the converse to theorem \ref{thmHuyghens} holds true -- as discussed in \cite{Markvorsen2016}.
\end{remark}

\section{Main results} \label{secMainThms}

We now show that the Hamilton orthogonality defined above is equivalent to a system of ODE equations for the rays of the net in question:

\begin{theorem} \label{thmMain}
Suppose $\gamma(s,t)$ is an $(N, t_{0})$-based unit fiber net in $M$. Then the following two conditions for $\gamma$ are equivalent:\\ \\
A: The \emph{energy pre-extremal condition}:
\begin{equation} \label{eqCE}
\rho(s,t) \cdot \dot{{\gamma}}_{t}(s,t) =  \sum_{j}\left(\ddot{\gamma}_{t}^{\,j}(s,t) + 2G^{j}(\dot{\gamma}_{t}(s,t)) + N_{0}^{j}(\dot{\gamma}_{t}(s,t))\right)\partial_{j}\, \, ,
\end{equation}
for some function $\rho$ of $s \in \, N$ and $t \in \, ]t_{0}, T[$,
and:\\ \\
B: The \emph{Hamilton orthogonality condition}:
\begin{equation} \label{eqHO}
\dot{\gamma}_{s}(s,t) \bot_{F_{t}} \dot{\gamma}_{t}(s,t) \quad \textrm{for all} \quad s \in N \, , \, t\, \in \, ]t_{0}, T[  \quad. \\ \\
\end{equation}
\end{theorem}

\begin{definition}
The unit fiber net satisfying one, hence both, of the conditions \eqref{eqHO} and \eqref{eqCE} is called \emph{the $(N, t_{0})$-based WF-net} in $(M, L)$.
\end{definition}

Before proving the theorem we first observe, that
standard ODE theory gives existence and uniqueness of a net $\gamma(s,t)$ which satisfies the equations \eqref{eqCE} and \eqref{eqHO}:

\begin{proposition} \label{propExist}
Suppose $N$ does not curve too much in the direction of its $n$-field at time $t_{0}$, i.e.  the hypersurface  $N$ has bounded second fundamental form in $M$ with respect to $F_{t_{0}}$. Then there exists a unique $(N, t_{0})$-based $WF$-net $\gamma(s,t)$ for $t \in \, ]t_{0}, T[$ for sufficiently small $T$.
\end{proposition}

\begin{remark}
The rays of a WF-net are usually not extremals for the $L$-energy because $\rho(s,t) =0$ is ususally not compatible with the unit speed condition in \eqref{eqWFray} in definition \ref{defNet}. Moreover, it is important to note that a given $(N, t_{0})$-based WF net may be quite different from an $(N, t_{1})$-based WF net for different initial times $t_{0}$ and $t_{1}$ -- see example \ref{exampZermelo}.
\end{remark}

\begin{proof}[Proof of theorem \ref{thmMain}]
 For transparency, and mostly in order to relate directly to the key 2D examples that we display below, we will assume that $M = \mathbb{R}^{2}$, i.e. $M$ is identical to its chart with coordinates $(u,v)$ in $\mathbb{R}^{2}$. This is done without much lack of generality since in the general setting we are only concerned with the semi-local aspects of the $N$-based nets in question. Moreover, the expressions and arguments in this proof generalize easily to higher dimensions.

Suppose first that the net satisfies \eqref{eqCE}.
Since each ray in the net by assumption  has unit $F$-speed, the specific variation $V = \dot{\gamma}_{s}$ determined by the rays in the net will give $\mathcal{L}^{\prime} = 0$ for all values of $b$ in the variational formula \eqref{eqLderiv} because each ray in the variation has constant length:
\begin{equation} \label{eqLderiv2}
\begin{aligned}
0 = \mathcal{L}'(0) &=
- \int_{t_{0}}^{b}\,g_{j\,k}\,\left(\ddot{c}^{\,j} + 2G^{j}(\dot{c}) + N_{0}^{j}(\dot{c}) \right)V^{k} \, dt \\
 &\phantom{aaaaaaa}+\sum_{i=1}^{m} \left[g_{j\,k}\,\dot{c}^{\,j}\, V^{k}\right]_{t_{i-1}}^{t_{i}} \quad .
\end{aligned}
\end{equation}
Since, also by assumption
\begin{equation}
\ddot{c}^{\,j} + 2G^{j}(\dot{c}) + N_{0}^{j}(\dot{c}) = \rho(t) \cdot \dot{c}^{j}
\end{equation}
for some function $\rho$ we get for all $b \in \, ]t_{0}, T[$:
\begin{equation} \label{eqInth}
0 = - \int_{t_{0}}^{b}\, m(t)\cdot \rho(t) \, dt + m(b) \quad ,
\end{equation}
where $m$ is shorthand for
\begin{equation} \label{eqDefh}
m(t) = \,g_{j\,k}\, \dot{c}^{j}(t) \, V^{k}(t)\quad .
\end{equation}
By assumption on the net  we have $m(t_{0}) = 0$ and from \eqref{eqInth} we get upon differentiation with respect to $b$:
\begin{equation}
m'(b) =  m(b)\cdot \rho(b) \quad \textrm{for all} \quad b \in \, ]t_{0}, T[ \quad.
\end{equation}
It follows that $m(t) = 0$ for all $t$ and therefore -- by \eqref{eqDefh} -- the variation vector field $V(t) = \dot{\gamma}_{s}(s,t)$ is $F$-orthogonal to $\dot{c}^{j}(t)\partial_{j} = \dot{\gamma}_{t}(s,t)$. Hence the net also satisfies equation \eqref{eqHO}.\\

Conversely, suppose that $\gamma(s,t)$ satisfies \eqref{eqHO}. Again we let $V$ denote any net-induced variation vector field along a given base ray $\gamma(s_{0},t) = c(t)$, which again satisfies the following equation for all $b$, since $V$ is assumed to be $F$-orthogonal to $\dot{c}$:
\begin{equation} \label{eqLderiv2}
0 = \mathcal{L}'(0) =
- \int_{t_{0}}^{b}\,g_{j\,k}\,\left(\ddot{c}^{\,j} + 2G^{j}(\dot{c}) + N_{0}^{j}(\dot{c}) \right)V^{k} \, dt  \quad .
\end{equation}
Since the integrand is thence identically $0$, the vector
\begin{equation}
\sum_{j}\left(\ddot{c}^{\,j} + 2G^{j}(\dot{c}) + N_{0}^{j}(\dot{c}) \right)\partial_{j}
\end{equation}
is $F$-orthogonal to the variation field $V$ and hence parallel to $\dot{c}(t) = \dot{\gamma}_{t}$. It follows, that the given net satisfying \eqref{eqHO} also satisfies \eqref{eqCE}.
\end{proof}

\begin{remark} \label{remEnPass}
We note en passant that if the Lagrangean $L$ is time-independent, i.e. if $(M, L)$ is a so-called regular scleronomic Lagrangean manifold, then $F$ is also time-independent and the rays of any WF net in $(M, L)$ are geodesics for the metric $F$. Indeed, in this case the WF net condition \eqref{eqCE} is precisely, that the rays are pre-geodesics of constant $F$-speed $1$, hence they are geodesics.
\end{remark}

\section{Frozen Finsler metrics} \label{secFrozen}

\begin{definition}
Let $\gamma(s,t)$ denote an $(N, t_{0})$-based  WF net in $(M, F)$ with the image $\mathcal{U} = \gamma(N \times \, ]t_{0}, T[\,)$. Then \emph{the frozen Finsler metric} $\widehat{F}$ on $\mathcal{U} \subset M$ is defined as the following \emph{time-independent (scleronomic) Finsler metric}:
\begin{equation}
\widehat{F}(u, v, x, y) = F(t(u,v), u, v, x, y) \quad \textrm{for all} \quad (u,v) \in \, \mathcal{U} \quad.
\end{equation}
The metric $\widehat{F}$ is called the \emph{frozen metric} associated to $F$ induced by the $(N, t_{0})$-based WF-net.
\end{definition}

We show that the rays of a WF net are geodesics of the frozen metric:

\begin{theorem} \label{thmFrozenGeodesics}
Let $\gamma(s,t)$ denote an $(N, t_{0})$-based WF-net in $(M, L)$ with associated frozen metric $\widehat{F}$ in the domain $\mathcal{U}$. Then the rays $\gamma(s_{0},t)$ of $\gamma$ are geodesic curves of $\widehat{F}$ so that the given WF-net in $(M, L)$ is also a WF-net in $(M, \widehat{F})$.
\end{theorem}

\begin{proof}
This follows immediately from the fact that the given WF-net of $F$ is also a WF-net of $\widehat{F}$ since the Hamilton orthogonality conditions are locally the same with respect to both metrics in the net which itself coordinates the 'freeze' of $F$ to $\widehat{F}$. \\

We now obtain the same result from the respective first variation formulas, which is not surprising in view of their role in the establisment above of theorem \ref{thmMain}.
The freezing time function $t=t(u,v)$ has partial derivatives which satisfy:
\begin{equation}
1 = \frac{\partial t}{\partial u}\,\dot{u}(t) + \frac{\partial t}{\partial v}\,\dot{v}(t)
\end{equation}
Moreover, since the gradient of $t$ in the Euclidean coordinate domain $\mathbb{R}^{2}$ is Euclidean-orthogonal to the frontal
\begin{equation} \label{eqTimeFrontal}
\eta_{t_{1}} =  \{ (u,v) \in \mathcal{U} \, | \,  t(u,v) = t_{1}\, \} \quad .
\end{equation}
we have:
\begin{equation}
0 = \frac{\partial t}{\partial u}\,W^{1}(u_{0},v_{0}) + \frac{\partial t}{\partial v}\,W^{2}(u_{0},v_{0}) \quad ,
\end{equation}
for any vector $W \in T_{(u_{0},v_{0})}\,\mathcal{U}$ which is tangent to the level set $\eta_{t_{0}}$ in $\mathbb{R}^{2}$.
In consequence we have for each vector $W$ which is $F$-orthogonal to $\dot{\gamma}$ (and therefore tangent to the $t_{0}$ level):
\begin{equation} \label{eqGradientOrto}
g_{j\,k}[{F}^{2}]_{t}\,\frac{\partial{t}}{\partial u^{l}}\, g^{j\, l}\, W^{k} = \frac{\partial{t}}{\partial u^{k}}\, W^{k} = 0 \quad.
\end{equation}
Since $\gamma(s,t)$ is a Hamilton orthogonal net, equation \eqref{eqGradientOrto}  means that the vector with coordinates
$g^{j\,i}[{F}^{2}]_{t}\,\frac{\partial{t}}{\partial u^{i}}$ is proportional to $\dot{\gamma}$.

The chain rule implies:
\begin{equation}
[\widehat{F}^{2}]_{u^{l}} = [{F}^{2}]_{u^{l}} + [{F}^{2}]_{t}\frac{\partial{t}}{\partial u^{l}} \quad ,
\end{equation}
so that we also have:
\begin{equation}
\begin{aligned}
&[\widehat{F}^{2}]_{u^{k}\, x^{l}}\,x^{k} - [\widehat{F}^{2}]_{u^{l}} \\
&= [{F}^{2}]_{u^{k}\, x^{l}}\,x^{k} + [{F}^{2}]_{t\, x^{l}}\,x^{k}\,\frac{\partial{t}}{\partial u^{k}}
- [{F}^{2}]_{u^{l}} - [{F}^{2}]_{t}\,\frac{\partial{t}}{\partial u^{l}} \\
&= [{F}^{2}]_{u^{k}\, x^{l}}\,x^{k} + [{F}^{2}]_{t\, x^{l}}
- [{F}^{2}]_{u^{l}} - [{F}^{2}]_{t}\,\frac{\partial{t}}{\partial u^{l}}  \quad .
\end{aligned}
\end{equation}

We insert these informations into the expression $\ddot{\gamma}^{\,j} + 2\widehat{G}^{j}$ for the time-independent metric $\widehat{F}$ and obtain:

\begin{equation}
\begin{aligned}
 \ddot{\gamma}^{\,j} + 2\widehat{G}^{j}
 &= \ddot{\gamma}^{\,j} + \left(\frac{1}{2}\right)\widehat{g}^{i\,l}\left([ \widehat{F}^{2}]_{u^{k}\,x^{l}}x^{k} - [ \widehat{F}^{2}]_{u^{l}} \right) \\
 &= \ddot{\gamma}^{\,j} + \left(\frac{1}{2}\right){g}^{i\,l}\left(
 [{F}^{2}]_{u^{k}\, x^{l}}\,x^{k} + [{F}^{2}]_{t\, x^{l}}
 - [{F}^{2}]_{u^{l}} \right)\\
 &\phantom{\ddot{\gamma}^{\,j} + + }- \left(\frac{1}{2}\right){g}^{i\,l}[{F}^{2}]_{t}\,\frac{\partial{t}}{\partial u^{l}}\\
&=\ddot{\gamma}^{\,j} + 2{G}^{j} + N_{0}^{j} - \left(\frac{1}{2}\right){g}^{i\,l}[{F}^{2}]_{t}\,\frac{\partial{t}}{\partial u^{l}}  \quad,
 \end{aligned}
\end{equation}
which is the L-extremal 'curvature' of $\gamma$ with respect to $F$ -- except for the last term which we know is proportional to $\dot{\gamma}_{t}$. By assumption on the WF-net, the vector $\sum_{j}(\ddot{\gamma}^{\,j} + 2{G}^{j} + N_{0}^{j})\partial_{j}$  is  proportional to $\dot{\gamma}_{t}$, so that $\sum_{j}(\ddot{\gamma}^{\,j} + 2\widehat{G}^{j})\partial_{j}$ is also proportional to $\dot{\gamma}_{t}$. Hence $\gamma$ is a pre-geodesic curve with respect to $\widehat{F}$. Since the length of $\dot{\gamma}_{t}$ is constrained
 to $1$ by the WF net condition, the curve $\gamma$ must be a geodesic in the metric $\widehat{F}$ -- see e.g. \cite[Exercise 5.3.2]{BCS}.
\end{proof}


\section{Rheonomic Randers metrics in $\mathbb{R}^{2}$} \label{secRanders}


A rheonomic Randers metric in $M = \mathbb{R}^{2}$ is  represented by its instantaneous \emph{elliptic} indicatrix fields.
The representing ellipse field $\mathcal{I}_{t,p} = E_{(t,u,v)}$ is parametrized as follows in the tangent space basis $\{\partial_{u}, \partial_{v}\}$ at $(u,v)$ in the parameter domain:
\begin{equation}
E_{(t,u,v)}(\psi) =  R_{\theta(t,u,v)}\left(\left[
                           \begin{array}{c}
                             a(t,u,v)\cos(\psi) \\
                             b(t,u,v)\sin(\psi) \\
                           \end{array}
                         \right] + \left[
                                     \begin{array}{c}
                                       c_{1}(t,u,v) \\
                                       c_{2}(t,u,v) \\
                                     \end{array}
                                   \right]\right)
\end{equation}
where $R_{\theta(t,u,v)}$ denotes the rotation in the tangent plane at $(u,v)$ by the angle $\theta(t,u,v)$ in the \emph{clock-wise direction}, see figure \ref{figDropField3}:
\begin{equation}
R_{\theta(t,u,v)} = \left[
                    \begin{array}{cc}
                      \cos(\theta(t,u,v)) & \sin(\theta(t,u,v)) \\
                      -\sin(\theta(t,u,v)) & \cos(\theta(t,u,v)) \\
                    \end{array}
                  \right]
\end{equation}
\begin{remark}
The choice of orientation of $\theta$ (turning in the clockwise direction) is customary in the field of wildfires, see \cite{Richards1990}, \cite{Markvorsen2016}.
\end{remark}

The translation vector $C(t,u,v)= (c_{1}(t,u,v), c_{2}(t,u,v))$ must always be assumed to be sufficiently small so that the resulting rotated and translated ellipse contains the origin of the tangent plane of the parameter domain at the point $(u,v)$, i.e. so that the ellipse with its origin becomes a pointed oval in the sense of general Finsler indicatrices.\\

The time-dependent Finsler metric induced from time-dependent Zermelo data is now obtained as follows -- see \cite{BRS}, \cite{Markvorsen2016}:
Let $h$ denote the Riemannian metric with the following component matrix with respect to the standard basis $\{\partial_{u} , \partial_{v} \}$ in every $T_{(u,v)}\mathbb{R}^{2}$:
\begin{equation} \label{eqZermelo}
h = \frac{1}{a^{2}b^{2}} \left[
      \begin{array}{cc}
       a^{2}\sin^{2}(\theta) + b^{2}\cos^{2}(\theta)& (a^2 - b^{2})\sin(\theta)\cos(\theta) \\
       (a^2 - b^{2})\sin(\theta)\cos(\theta)  & a^{2}\cos^{2}(\theta) + b^{2}\sin^{2}(\theta) \\
      \end{array}
    \right] \quad .
\end{equation}

 For elliptic template fields the direct way to the Finsler metric from the Zermelo data is as follows:
 Suppose for example that we are given ellipse field data $a(t,u,v)$, $b(t,u,v)$, $C(t,u,v) = (c_{1}(t,u,v), c_{2}(t,u,v))$, and $\theta(t,u,v)$.
Then the corresponding Finsler metric $F$ is determined by the following expression, where $V = (x,y)$ denotes any vector in  $\in T_{(u,v)}\mathbb{R}^{2}$ -- see e.g. \cite[Section 1.1.2]{BRS} and \cite{Markvorsen2016}:
\begin{equation}
\begin{aligned}
F &= F_{t} = F(t, p, V) = F(t, u,v,x,y) \\
&= \left(\frac{\sqrt{\lambda\cdot h(V,V) + h^{2}(V,C)}}{\lambda}\right) - \left(\frac{h(V,C)}{\lambda} \right) \, ,
\end{aligned}
\end{equation}
where
\begin{equation}
\lambda = 1 - h(C,C) > 0 \quad .
\end{equation}


\section{Richard's equations} \label{secRichards}

Richards observed in \cite{Richards1990} that Huyghens' principle for a spread phenomenon (in casu the spread of wildfires) in the plane together with a background indicatrix field of Zermelo type as discussed above forces the frontals of the spread to satisfy the PDE system of differential equations in theorem \ref{thmRichardsHamilton} below.

\begin{theorem} \label{thmRichardsHamilton}
We consider an $(N, t_{0})$ based WF net $\gamma(s, t)$ in $\mathbb{R}^{2}$ with a given ellipse template field for a Finsler metric $F$ with Zermelo equivalent data $a(t,u,v)$, $b(t,u,v)$, $C(t,u,v) = (c_{1}(t,u,v), c_{2}(t,u,v))$, and $\theta(t,u,v)$. Suppose that the net satisfies Huyghens' principle as by definition \ref{defHuyghensPrincip}. Then the net is determined by the following equations for the partial derivatives $\dot{\gamma}_{s}(s,t) = (\dot{u}_{s}, \dot{v}_{s})$ and $\dot{\gamma}_{t}(s,t) = (\dot{u}_{t}, \dot{v}_{t})$, where now, as indicated above, the control coefficients $a$, $b$, $c_{1}$, $c_{2}$, and $\theta$ are all allowed to depend on both time $t$ and position $(u,v)$:
\begin{eqnarray} \label{eqRichardsHamilton}
\dot{u}_{t} & = & \frac{a^{2}\cos(\theta)\left( \dot{u}_{s}\sin(\theta) + \dot{v}_{s}\cos(\theta) \right) - b^{2}\sin(\theta)\left( \dot{u}_{s}\cos(\theta) - \dot{v}_{s}\sin(\theta)  \right)}{\sqrt{a^{2}\left(\dot{u}_{s}\sin(\theta) + \dot{v}_{s}\cos(\theta)\right)^{2} + b^{2}\left(\dot{u}_{s}\cos(\theta) - \dot{v}_{s}\sin(\theta)\right)^{2}}} \nonumber \\
 & & + c_{1}\,\cos(\theta) + c_{2}\, \sin(\theta) \label{eq1} \quad , \quad \textrm{and}  \nonumber   \\ \nonumber\\
\dot{v}_{t} & = & \frac{-a^{2}\sin(\theta)( \dot{u}_{s}\sin(\theta) + \dot{v}_{s}\cos(\theta) ) - b^{2}\cos(\theta)( \dot{u}_{s}\cos(\theta) - \dot{v}_{s}\sin(\theta)  )}{\sqrt{a^{2}(\dot{u}_{s}\sin(\theta) + \dot{v}_{s}\cos(\theta))^{2} + b^{2}(\dot{u}_{s}\cos(\theta) - \dot{v}_{s}\sin(\theta))^{2}}} \nonumber  \\
& & - c_{1}\,\sin(\theta) + c_{2}\, \cos(\theta) \label{eq2}  \nonumber \quad .
\end{eqnarray}
\end{theorem}

For the very special and rare cases where the control coefficients $a$, $b$, $c_{1}$, $c_{2}$, and $\theta$ are assumed to depend only on time $t$ and not on the position $(u,v)$ Richards was able to find the following explicit analytic solution to the equations in theorem \ref{thmRichardsHamilton} -- see \cite[Section 4.1]{Richards1993a}:

\begin{theorem} \label{thmRichardsExplicit}
For a rheonomic 2D Randers metric in $\mathbb{R}^{2}$ with no spatial dependence the solution to the equations in theorem \ref{thmRichardsHamilton} are given by the following expressions for $\gamma(s,t) = (u(s,t), v(s,t))$ with $(u(0), v(0))= (u_{0}, v_{0})$:
\begin{equation*}
u(s,t) =  u_{0} + \int_{0}^{t}\, f(r)  \,\,dr
+ \int_{0}^{t}\, c_{2}(r)\sin(\theta(r)) + c_{1}(r)\cos(\theta(r)) \,\,dr \quad ,
\end{equation*}
where
\begin{equation*}
f(r) = \frac{a^{2}(r)\cos(\theta(r))\cos(\theta(r) + s) + b^{2}(r)\sin(\theta(r))\sin(\theta(r)+s)}{\sqrt{a^{2}(r)\cos^{2}(\theta(r) + s) + b^{2}(r)\sin^{2}(\theta(r) + s)}} \quad ,
\end{equation*}
and
\begin{equation*}
v(s,t) =  v_{0} + \int_{0}^{t}\, g(r) \,\,dr
+ \int_{0}^{t}\, c_{2}(r)\cos(\theta(r)) - c_{1}(r)\sin(\theta(r)) \,\,dr \quad ,
\end{equation*}
where
\begin{equation*}
g(r) = \frac{-a^{2}(r)\sin(\theta(r))\cos(\theta(r) + s) + b^{2}(r)\cos(\theta(r))\sin(\theta(r)+s)}{\sqrt{a^{2}(r)\cos^{2}(\theta(r) + s) + b^{2}(r)\sin^{2}(\theta(r) + s)}}
\quad .
\end{equation*}
\end{theorem}

\begin{remark}\label{remConsist}
The solutions presented in theorem \ref{thmRichardsExplicit} -- and indirectly in theorem \ref{thmRichardsHamilton} -- can be shown directly to satisfy the unit $F$-speed condition and  the Hamilton orthogonality condition needed to form  $(p, 0)$-based WF nets with $p = (u_{0}, v_{0})$.
\end{remark}

We illustrate the construction of a frozen metric in a most simple example using the following rheonomic  metric:

\begin{example}\label{exampZermelo}
In $\mathcal{U}= \mathbb{R}^{2}$ we consider the time-dependent Riemannian metric for all $t\geq 0$:
\begin{equation}
F(t, u, v, x,y) = F(t, u^{1}, u^{2}, x^{1}, x^{2})
= \frac{\sqrt{x^2 + (y/2)^2}}{1+t}  \quad   .
\end{equation}
The corresponding Zermelo data are clearly the following:
\begin{equation}
\begin{aligned}
a(t,u,v) &= 1+t \\
b(t,u,v) &= 2 + 2t\\
C(t,u,v) &= (0, 0) \\
\theta(t,u,v) &= 0 \quad ,
\end{aligned}
\end{equation}

For the rheonomic metric $F$ we then have the following ingredients for the energy extremals and for the WF-ray equations:
\begin{equation}
\begin{aligned}
&\left[ F^{2}\right]_{u^{l}} = 0\\
&\left[ F^{2}\right]_{u^{k}\,x^{l}} = 0\\
&\left[ F^{2}\right]_{t} = -\,\frac{4x^{2} + y^{2}}{2(1 + t)^3}  \\
&\left[ F^{2}\right]_{t\,x^{1}} = \frac{-4x}{(1 + t)^3} \\
&\left[ F^{2}\right]_{t\,x^{2}} = \frac{-y}{(1 + t)^3} \\
&g_{1\,1} = \frac{1}{(1+t)^2}\\
&g_{2\,2} = \frac{1}{4(1+t)^2}\\
&g_{1\,2} = g_{2\,1} = 0\\
&G^{i} = 0\\
&N_{0}^{1}(t,u,v,x,y) = N_{0}^{1} = \frac{-2x}{1+t}\\
&N_{0}^{2}(t,u,v,x,y) = N_{0}^{2} = \frac{-2y}{1+t}\quad .
\end{aligned}
\end{equation}

The energy extremal equations are thence:
\begin{equation}
\begin{aligned}
\ddot{u}(t) + {N}_{0}^{1}(t, u(t), v(t), \dot{u}(t), \dot{v}(t)) &= 0 \\
\ddot{v}(t) + {N}_{0}^{2}(t, u(t), v(t), \dot{u}(t), \dot{v}(t)) &= 0  \quad ,
\end{aligned}
\end{equation}
or, equivalently:
\begin{equation}
\begin{aligned}
\ddot{u} -  \frac{2\dot{u}}{1+t} &= 0 \\
\ddot{v} -  \frac{2\dot{v}}{1+t} &= 0 \quad .
\end{aligned}
\end{equation}

The solutions to these full energy extremal equations issuing from $(u_{0}, v_{0}) = (0, 0)$ at time $t=0$ in the direction of the (non-unit vector) $(\dot{u}(0), \dot{v}(0)) = (\cos(s), \sin(s))$ are the following parametrized radial half lines:
\begin{equation}
\begin{aligned}
u(s,t) &= \left(\frac{2}{3}t^{3} + 2t^{2} + 2t \right) \frac{\cos(s)}{\sqrt{1 + 3\cos^{2}(s)}}\\
v(s,t) &=  \left(\frac{2}{3}t^{3} + 2t^{2} + 2t\right) \frac{\sin(s)}{\sqrt{1 + 3\cos^{2}(s)}}\quad .
\end{aligned}
\end{equation}
As expected, the rays of this solution do not have constant unit speed:
\begin{equation}
F(t, u, v, \dot{u}, \dot{v}) = 1+t \quad .
\end{equation}

The WF net rheonomic equations \eqref{eqCE}, however, give the correct solutions with ${F}(t, u, v, \dot{u}, \dot{v}) = 1$ for all $s$ and $t$:
\begin{equation} \label{eqCEsol1}
\begin{aligned}
u(s,t) &= \left(t^{2} + 2t\right) \frac{\cos(s)}{\sqrt{1 + 3\cos^{2}(s)}}\\
v(s,t) &=  \left(t^{2} + 2t\right) \frac{\sin(s)}{\sqrt{1 + 3\cos^{2}(s)}} \quad .
\end{aligned}
\end{equation}

From this solution we can now extract the time function:
\begin{equation}
t = t(u,v) =  -1 + \sqrt{1 + \sqrt{4u^{2} + v^{2}}}\quad,
\end{equation}
and insert it into the rheonomic Finsler metric to obtain the corresponding frozen Finsler metric $\widehat{F}$ as follows:
\begin{equation}
\widehat{F}(u,v,x,y) = F(t(u,v),u,v,x,y) = \frac{\sqrt{4x^2 + y^2}}{2\sqrt{1 + \sqrt{4u^{2} + v^{2}}}} \quad.
\end{equation}
For this frozen metric we have correspondingly for its WF-net geodesic equations:
\begin{equation}
\begin{aligned}
&\widehat{G}^{1}(u,v,x,y) = \frac{u\,(-4x^{2}+ y^{2}) - 2v\,x\,y}{4\left(1 + \sqrt{4u^{2} + v^{2}}\right)\sqrt{4u^{2} + v^{2}}}\\
&\widehat{G}^{2}(u,v,x,y) = \frac{v\,(4x^{2}- y^{2}) - 8u\,x\,y}{4\left(1 + \sqrt{4u^{2} + v^{2}}\right)\sqrt{4u^{2} + v^{2}}} \quad .
\end{aligned}
\end{equation}

The 'frozen' geodesics for $\widehat{F}$ satisfies the equations:
\begin{equation}
\begin{aligned}
\ddot{u}(t) + 2\widehat{G}^{1}(u(t), v(t), \dot{u}(t), \dot{v}(t)) &= 0 \\
\ddot{v}(t) + 2\widehat{G}^{2}(u(t), v(t), \dot{u}(t), \dot{v}(t)) &= 0 \quad .
\end{aligned}
\end{equation}

These equations are solved by the same rays as presented in equation \eqref{eqCEsol1}.
Also they are the same rays as those obtained from Richards' recipe in theorem \ref{thmRichardsExplicit} which are parametrized as follows:
\begin{equation} \label{eqRichardsol}
\begin{aligned}
\tilde{u}(\tilde{s},t) &= \left(\frac{t^{2}}{2} + t\right) \frac{\cos(\tilde{s})}{\sqrt{4 - 3\cos^{2}(\tilde{s})}}\\
\tilde{v}(\tilde{s},t) &=  \left(2t^{2} + 4t\right) \frac{\sin(\tilde{s})}{\sqrt{4 - 3\cos^{2}(\tilde{s})}} \quad .
\end{aligned}
\end{equation}
The only difference is due to a different choice of parametrization of directions from $p$. Indeed, if we transform $\tilde{s}$ to $s$ via the consistent equations
\begin{equation}
\begin{aligned}
 \frac{\cos(\tilde{s})}{\sqrt{4 - 3\cos^{2}(\tilde{s})}} &= \frac{2\cos(s)}{\sqrt{1 + 3\cos^{2}(s)}}  \\
  \frac{2\sin(\tilde{s})}{\sqrt{4 - 3\cos^{2}(\tilde{s})}} &= \frac{\sin(s)}{\sqrt{1 + 3\cos^{2}(s)}}  \quad ,
 \end{aligned}
\end{equation}
then the two WF net solutions \eqref{eqRichardsol} and \eqref{eqCEsol1} do agree.\\

Now, in comparison, we may consider instead the WF net geodesics obtained by starting at $(0,0)$ at the \emph{later time} $t_{0}=1$ in the same rheonomic metric $F$ as above. A calculation along the same lines as above now gives the new time function:
\begin{equation}
t = t(u,v) =  -1 + \sqrt{4 + \sqrt{4u^{2} + v^{2}}} \quad .
\end{equation}
Insertion of this new time function into the given rheonomic metric $F$ now gives the new frozen metric:
\begin{equation}
\widetilde{F}(u,v,x,y) = \frac{\sqrt{4x^{2} + y^{2}}}{2\sqrt{4 + \sqrt{4u^{2} + v^{2}}}}
\end{equation}
with the following frozen geodesics -- with $(u(1), v(1)) = (0,0)$:
\begin{equation}
\begin{aligned}
\widetilde{\gamma}(s,t) &= (u(s,t), v(s,t))\\
&= \left( \frac{(t^2 + 2t -3)\cos(s)}{\sqrt{1 + 3\cos^{2}(s)}}\, , \, \, \frac{(t^2 + 2t -3)\sin(s)}{\sqrt{1 + 3\cos^{2}(s)}} \right) \quad .
\end{aligned}
\end{equation}
The geodesic curves of the frozen metric are the same straight line curves as before -- only now with a reparametrization in the $t$-direction. Observe that the new parametrization $t^2 + 2t -3$ is not simply obtained by inserting $t-1$ in place of $t$ in the old parametrization $t^2 + 2t$ . The frozen metric is clearly different -- with an extra $4$ in the denominator square root.
\end{example}

Two somewhat more complicated examples of rheonomic metrics in $\mathbb{R}^{2}$ are considered in the next example:

\begin{example}\label{exampZermeloSimp}
We present two simple cases of time-dependent Randers metrics in $\mathbb{R}^{2}$. Their respective time-dependent indicatrix fields are indicated in figures  \ref{figDropField3} and \ref{figDropField4}:
The Zermelo data for figure \ref{figDropField3} are as follows:
\begin{equation} \label{eqZermelo1}
\begin{aligned}
a(t,u,v) &= 1 \\
b(t,u,v) &= 2 + t/5\\
C(t,u,v) &= (0, 0) \\
\theta(t,u,v) &= ((t+5)+u-v)/20 \quad .
\end{aligned}
\end{equation}
The Zermelo data for figure \ref{figDropField4}, which depend on time only, are the same as for  figure \ref{figDropField3}, except that the $u$ and $v$ dependence has been removed (from the rotation angle $\theta$) so that:
\begin{equation} \label{eqZermelo2}
\begin{aligned}
a(t,u,v) &= 1\\
b(t,u,v) &= 2+t/5\\
C(t,u,v) &= (0,0)\\
\theta(t,u,v) &= (t+5)/20\quad .
\end{aligned}
\end{equation}
\end{example}

\begin{figure}[h!]
\centerline{
\includegraphics[width=50mm]{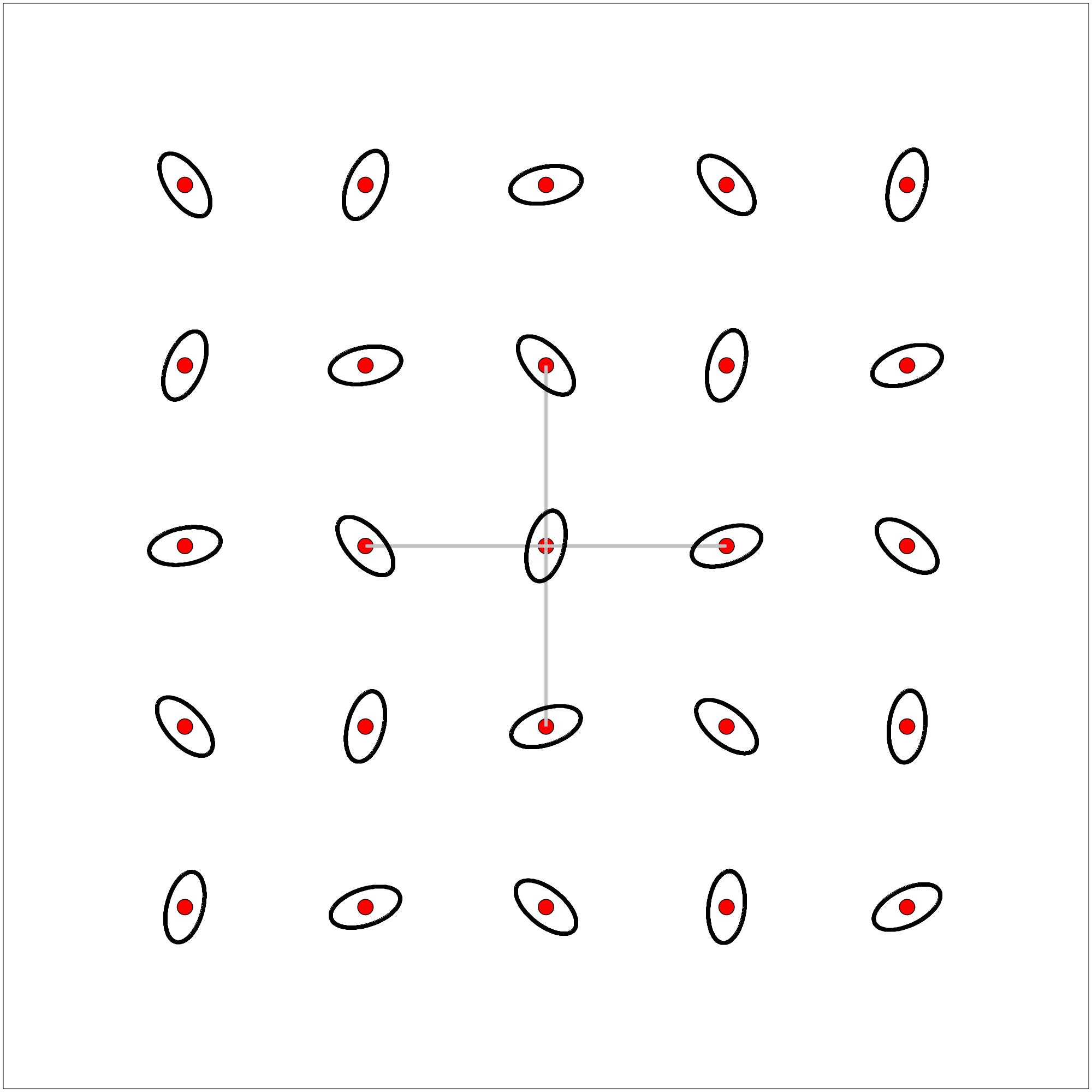}\quad
\includegraphics[width=50mm]{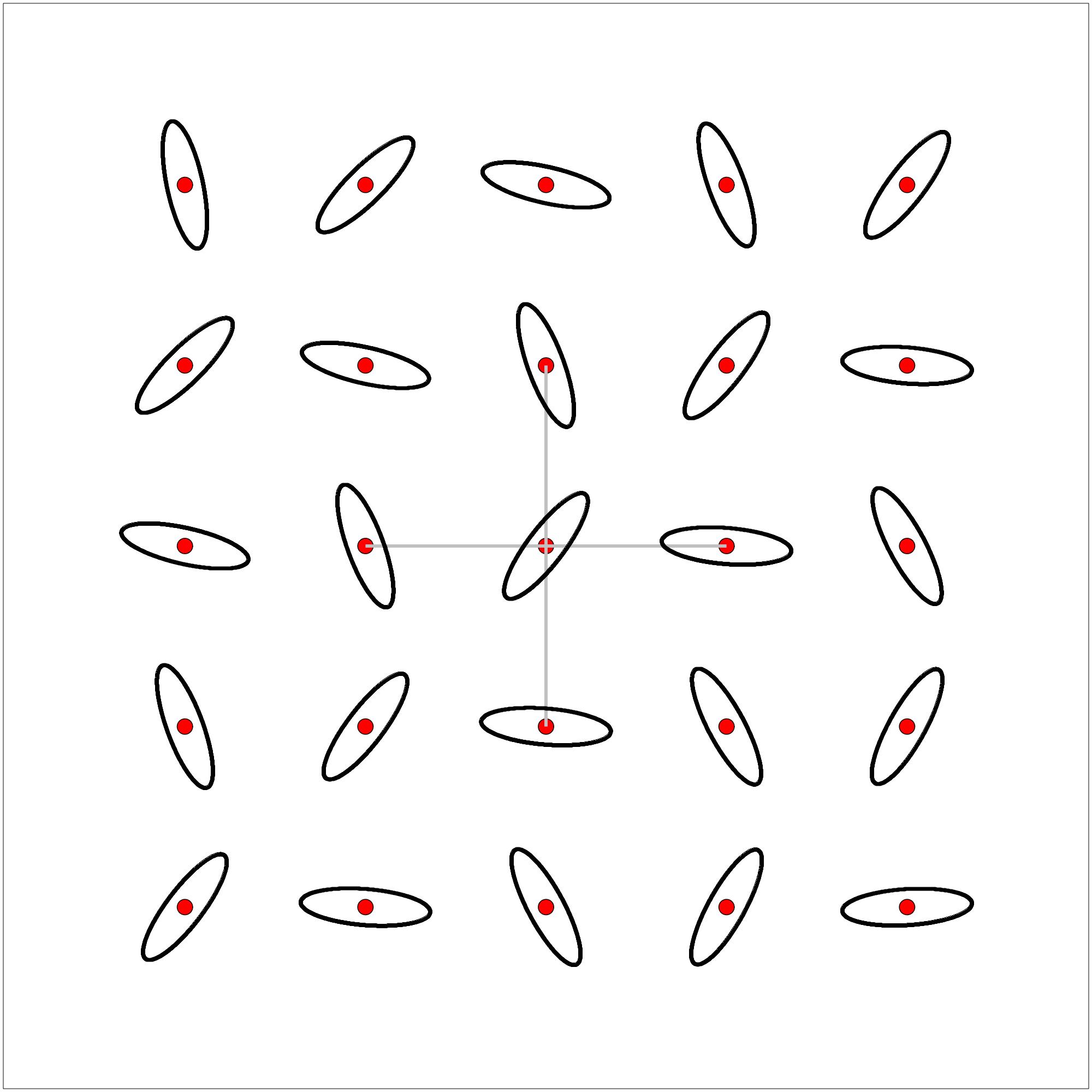}\quad
\includegraphics[width=50mm]{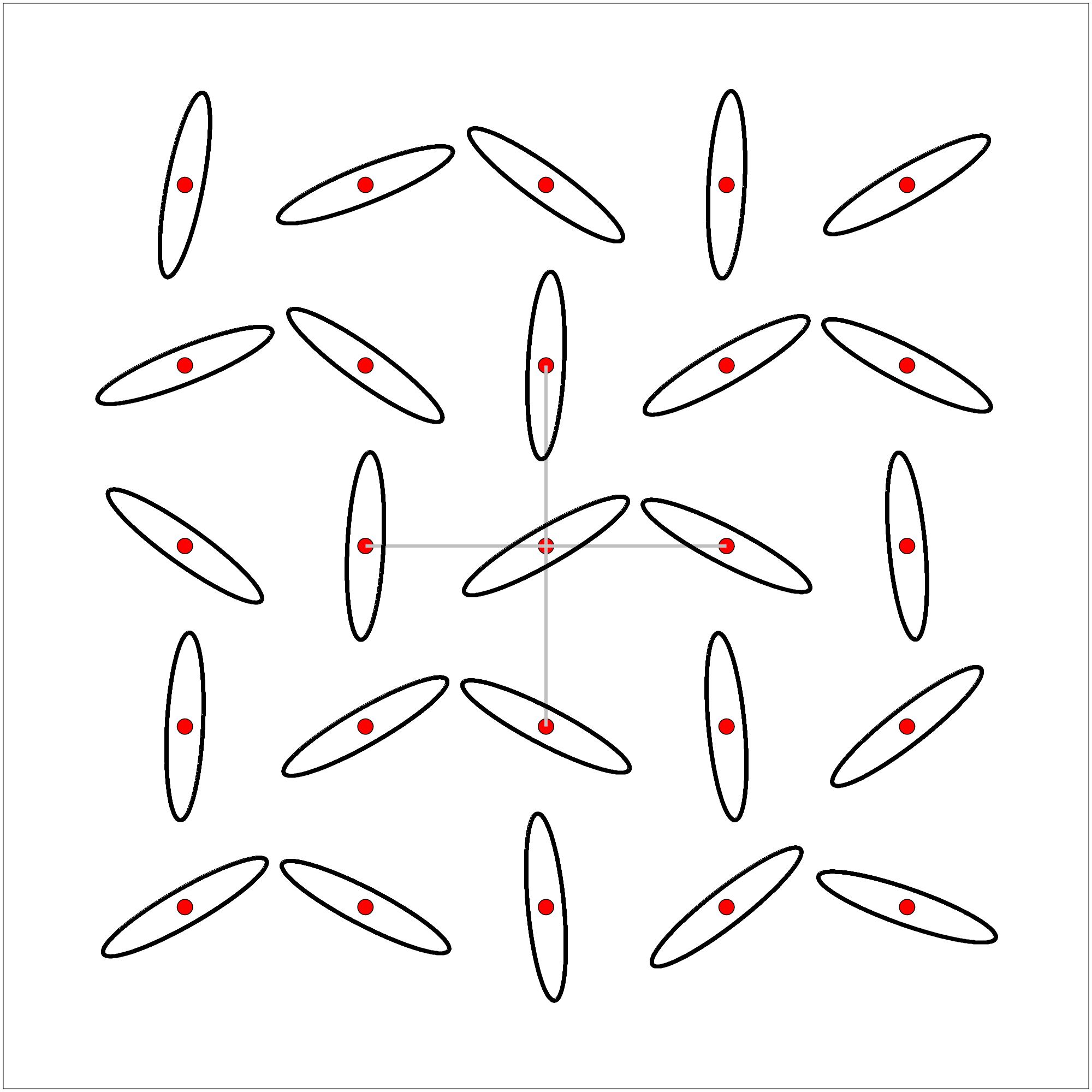} }
\begin{center}
\caption{{The indicatrix field from the Zermelo data in  \eqref{eqZermelo1} at three different times, $t=0$, $t=8$, and $t=16$ (from left to right). The spatial dependence is determined solely via the Zermelo rotation angle $\theta(t,u,v)$ in \eqref{eqZermelo1}.}} \label{figDropField3}
\end{center}
\end{figure}

\begin{figure}[h!]
\centerline{
\includegraphics[width=50mm]{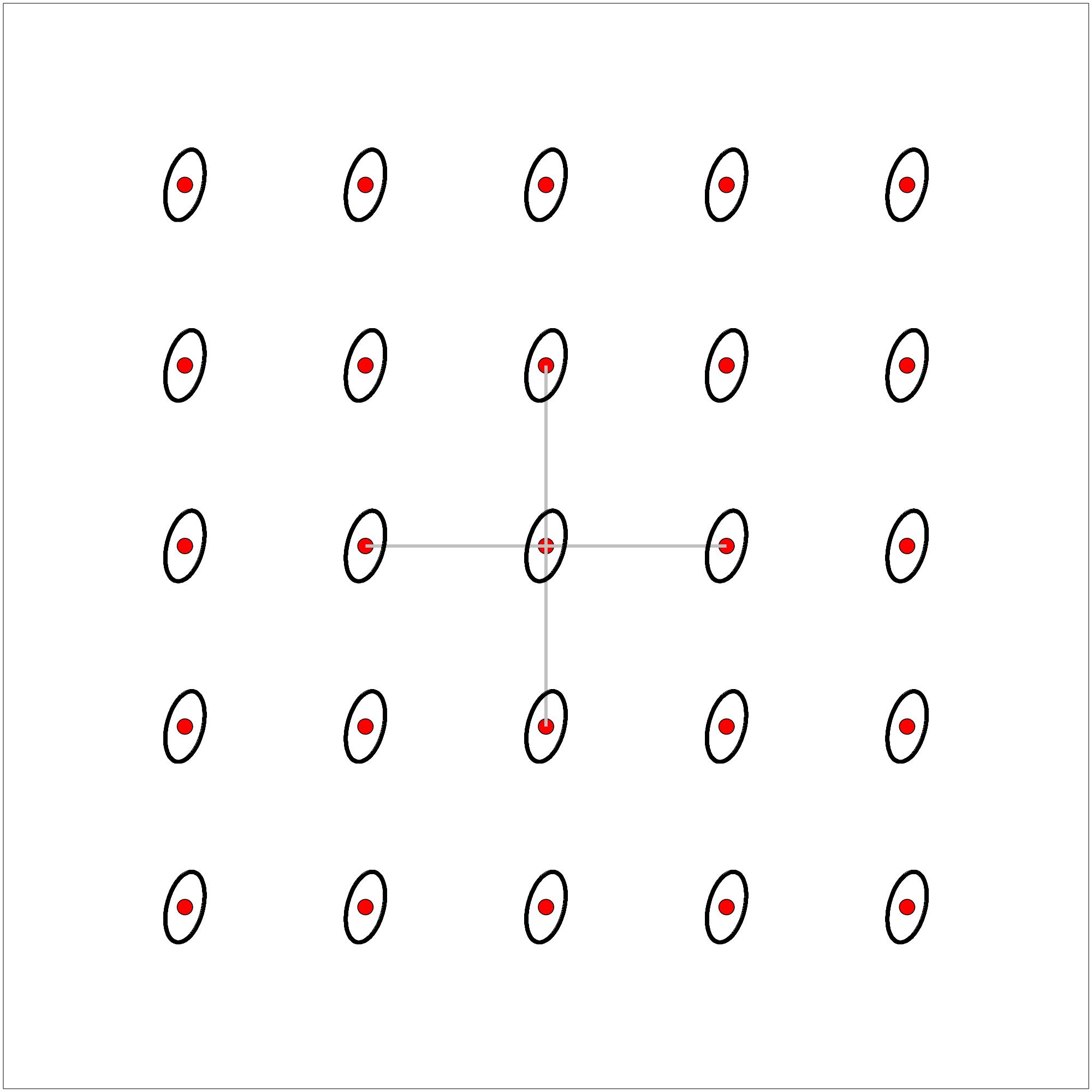} \quad
\includegraphics[width=50mm]{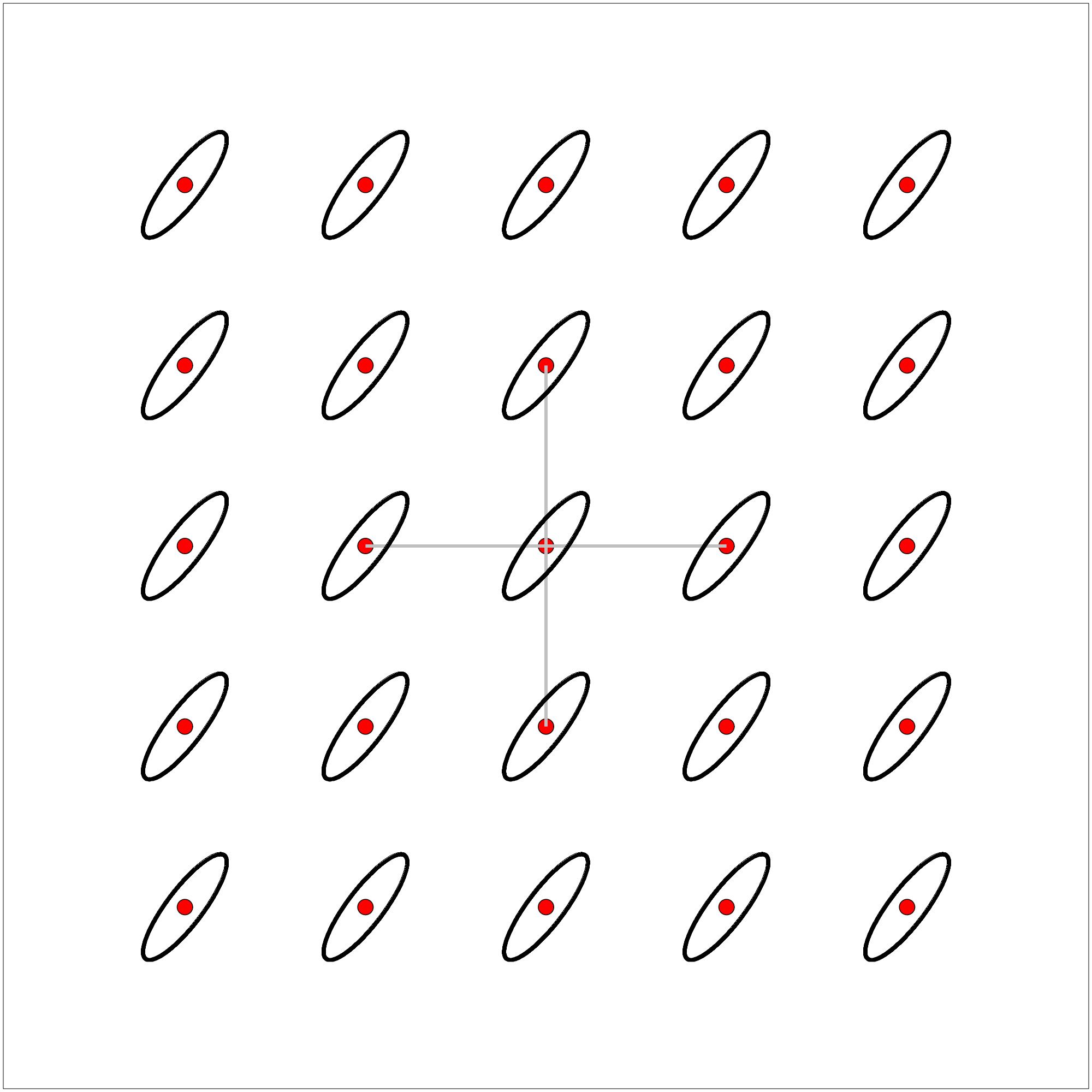}\quad
\includegraphics[width=50mm]{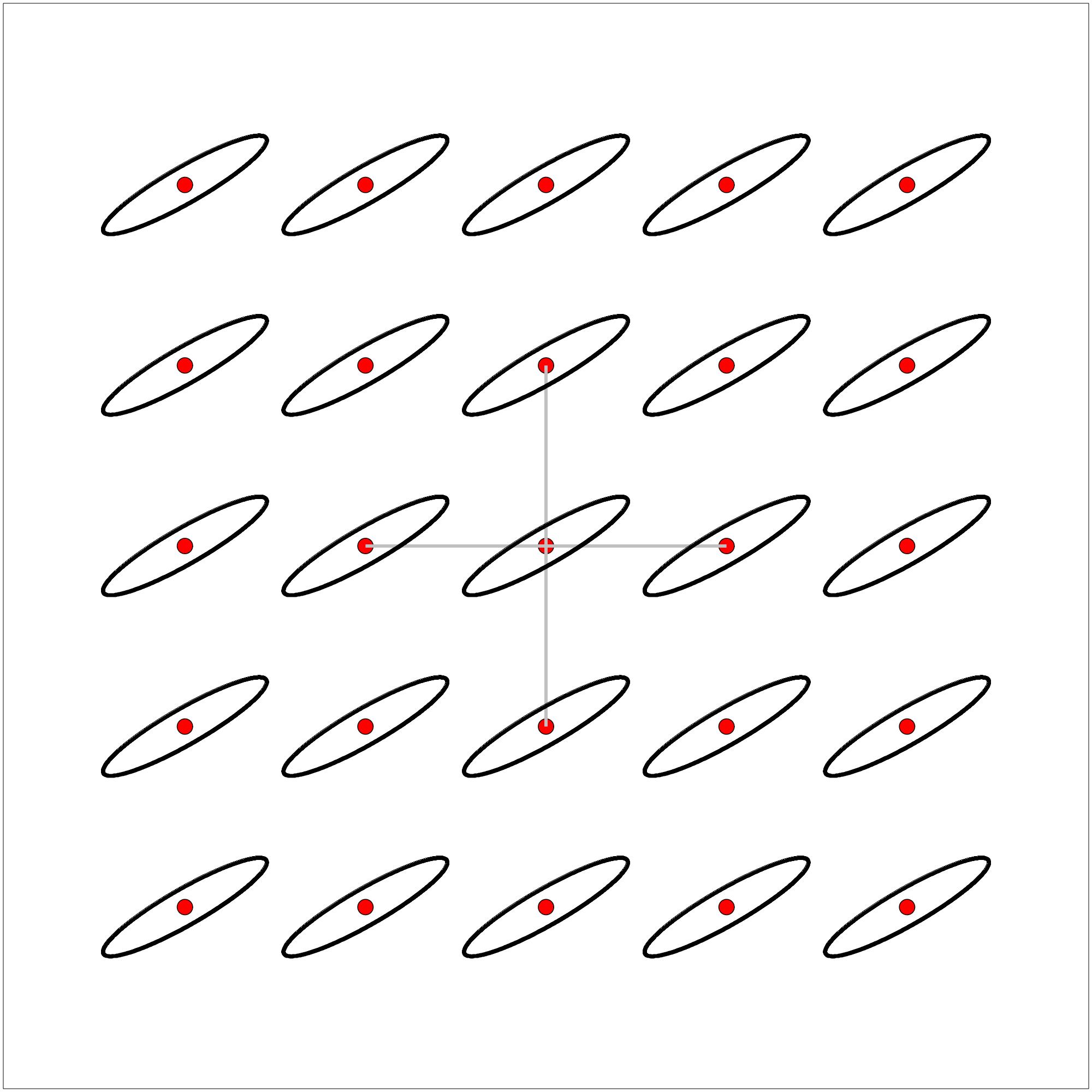} }
\begin{center}
\caption{The indicatrix field from the Zermelo data in \eqref{eqZermelo2} at three different times, $t=0$, $t=8$, and $t=16$ (from left to right). The field is clearly spatially constant in each time-slice.} \label{figDropField4}
\end{center}
\end{figure}

For these two relatively simple fields we obtain -- via numerical solutions -- the $(p, 0)$-based WF nets displayed in the figures \ref{figSolutionNet3} and \ref{figSolutionNet4}, respectively, with ignition at point $p = (0,0)$ at time $t=0$. The displayed frontals are then obtained at time values $t_{i} = 0.2\cdot 16 \cdot i$, where $i = 1, \cdots, 5$ so that the outermost frontal is the time-level set $\eta_{t_{5}} = \eta_{16}$. In figure \ref{figHuyghens34} is shown also the two sets of Huyghens droplets of duration $0.2\cdot 16$ ignited from points on the two respective level sets  $\eta_{t_{4}}$ at the corresponding time $t_{4}$. In the two cases under consideration the Huyghens droplets clearly show a tendency to envelope the outer frontal $\eta_{16}$ -- cf. remark \ref{remHuyghensRev}.

\begin{figure}[h]
\centerline{
\includegraphics[width=60mm]{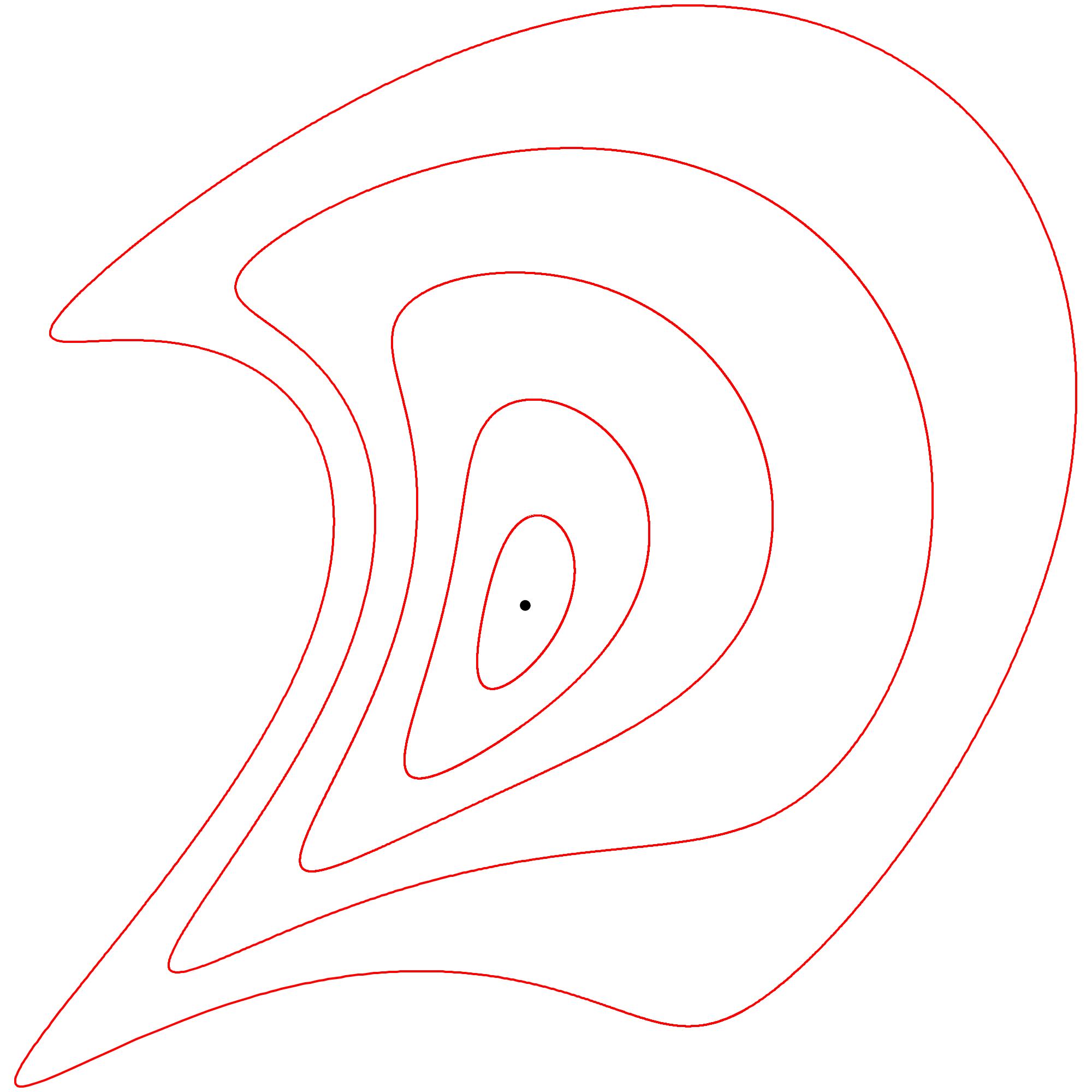}\quad \includegraphics[width=60mm]{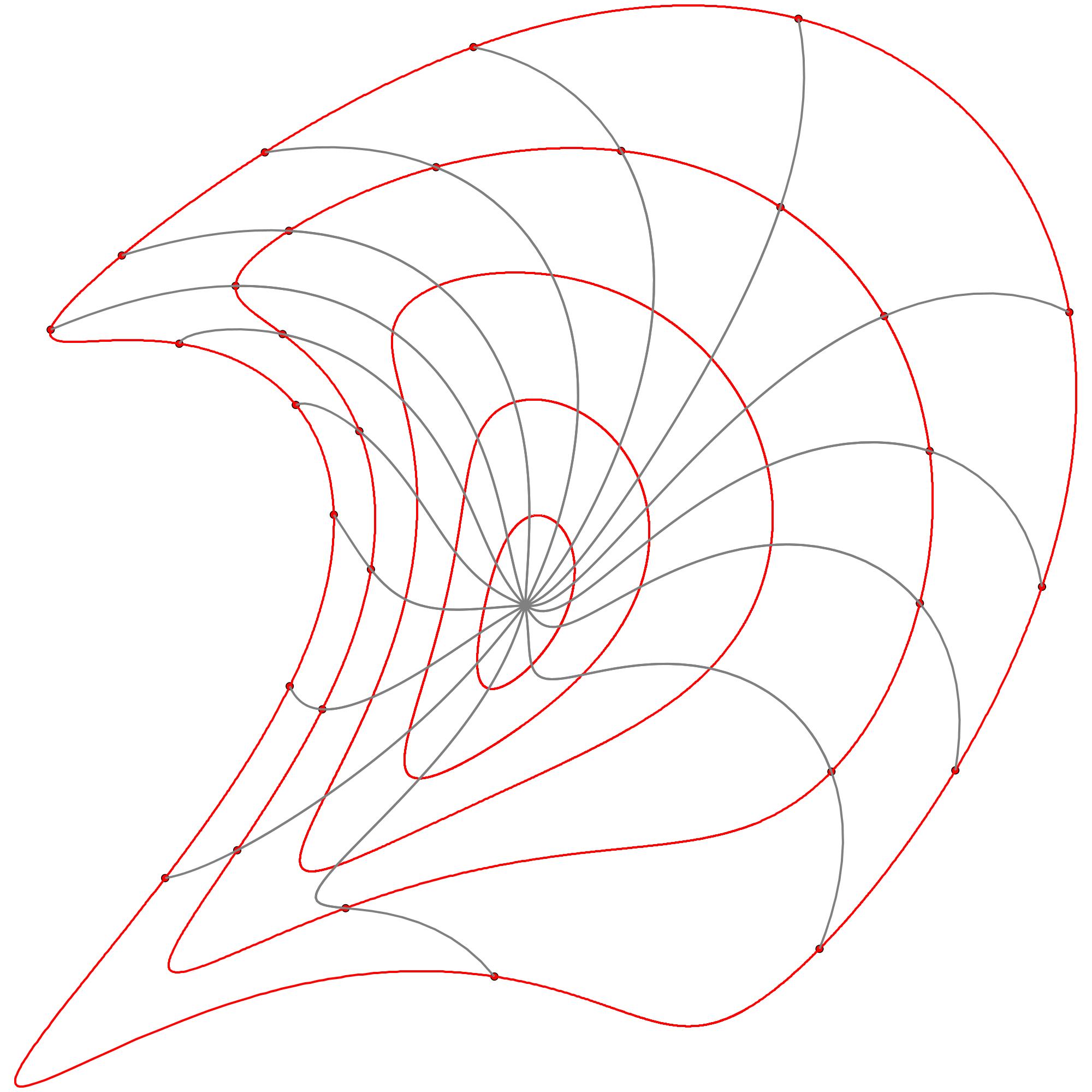}}
\begin{center}
\caption{{The WF net frontals and rays for the rheonomic field with Zermelo data \eqref{eqZermelo1} ignited at time $0$ at the point $(0,0)$ and running until time $T=16$ in steps of $0.2\cdot 16$.}} \label{figSolutionNet3}
\end{center}
\end{figure}

\begin{figure}[h]
\centerline{
\includegraphics[width=60mm]{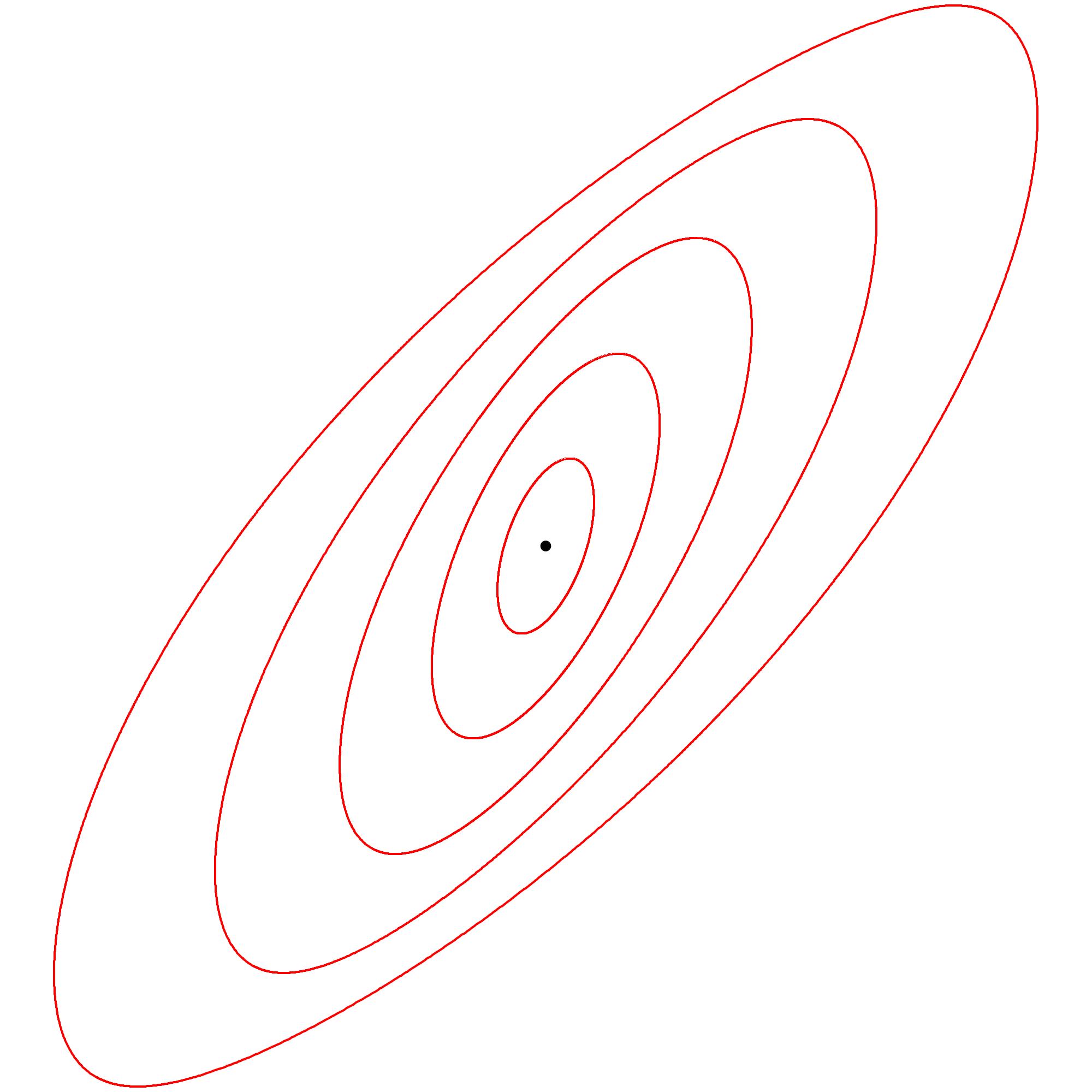} \quad \includegraphics[width=60mm]{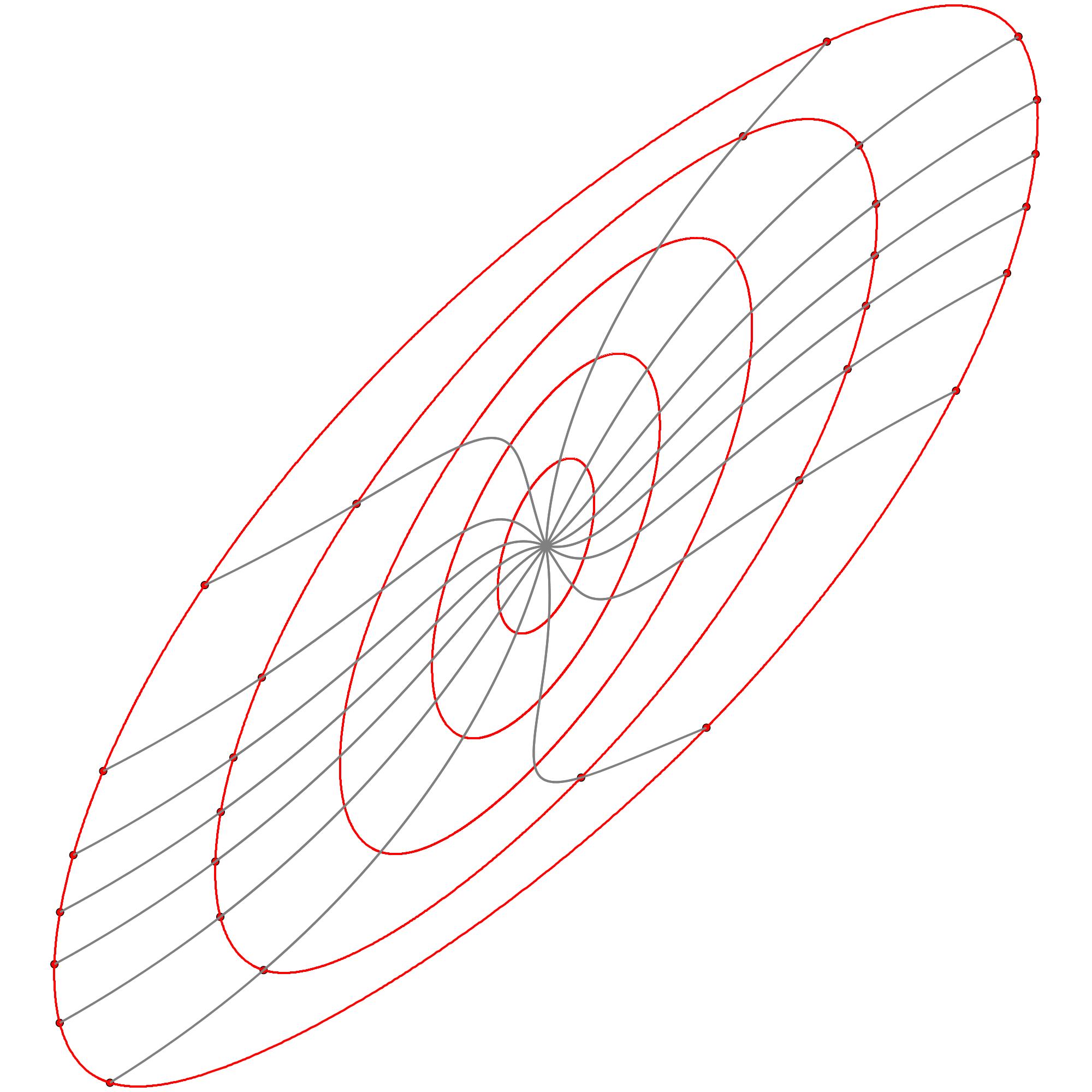}}
\begin{center}
\caption{{The WF net frontals and rays for the rheonomic field with Zermelo data \eqref{eqZermelo2} ignited at time $0$ at the point $(0,0)$ and running until time $T=16$ in steps of $0.2\cdot 16$.}} \label{figSolutionNet4}
\end{center}
\end{figure}

\begin{figure}[h]
\centerline{
\includegraphics[width=60mm]{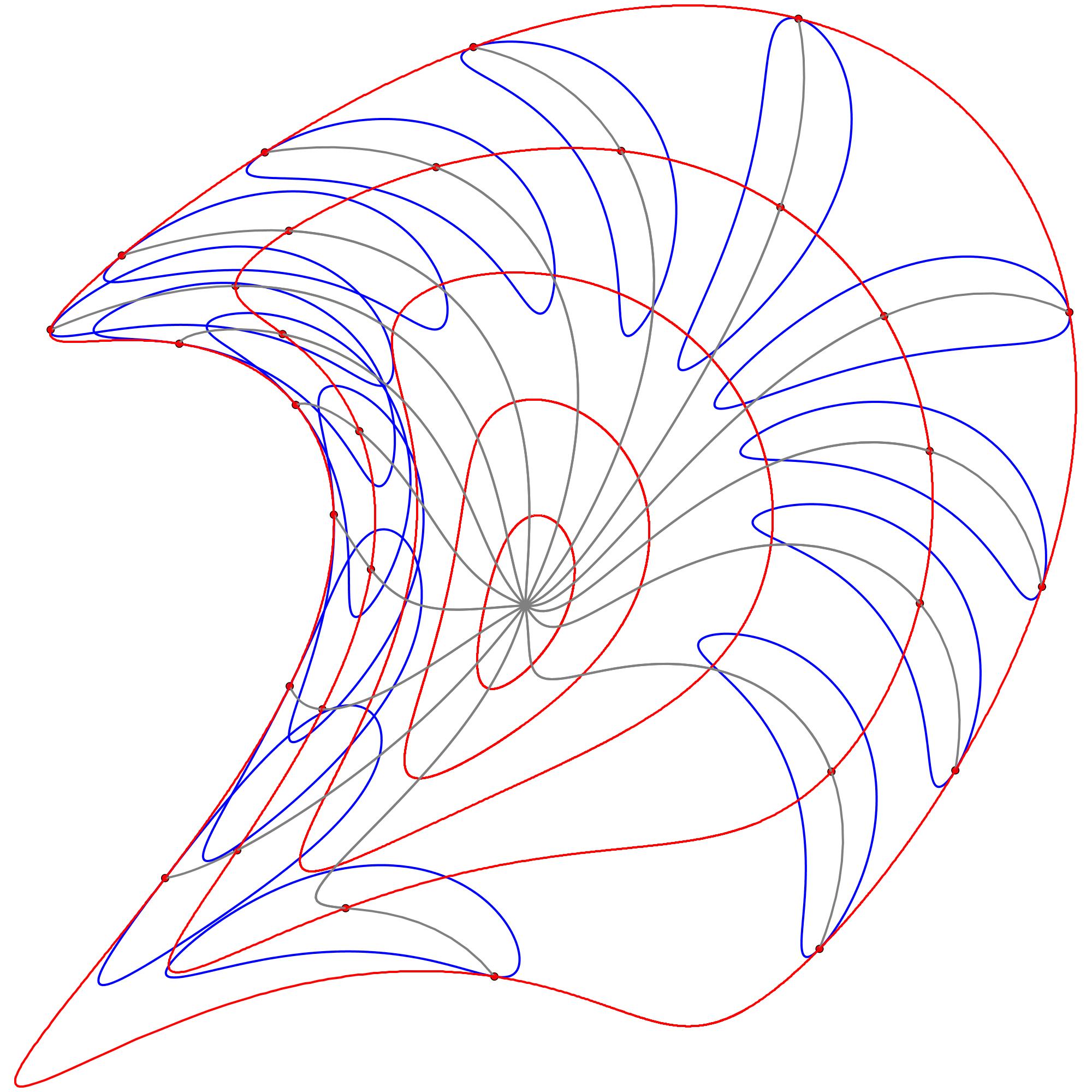} \quad \includegraphics[width=60mm]{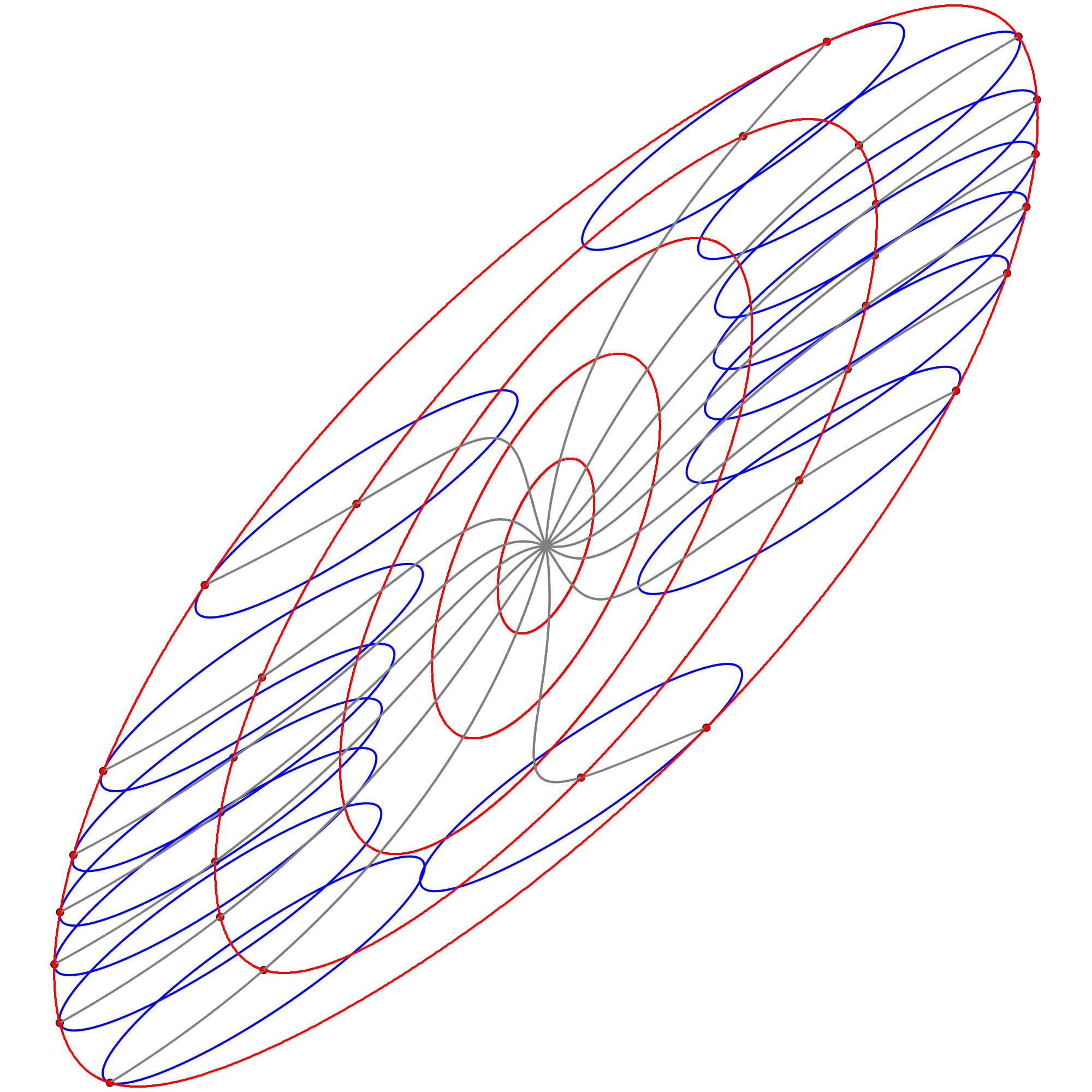}}
\begin{center}
\caption{{The Huyghens droplets of duration $0.2\cdot 16$ from the frontal $\eta_{0.8\cdot 16}$ seem to envelope the outermost frontal $\eta_{16}$ in both cases of Zermelo data \eqref{eqZermelo1} and \eqref{eqZermelo2} from example \ref{exampZermeloSimp}.}} \label{figHuyghens34}
\end{center}
\end{figure}


\vspace{2cm}

\bibliographystyle{plain}


\def\cprime{$'$} \def\cprime{$'$}


\end{document}